\documentclass[12pt]{article}
\usepackage{a4}
\usepackage{amsthm}
\usepackage{amsfonts}
\usepackage{amssymb}
\usepackage{amsmath}
\usepackage{algorithm2e}
\usepackage{cite}
\usepackage{epsfig}
\usepackage{xcolor}
\usepackage{makecell}
\usepackage{array}

\newtheorem{theorem}{Theorem}
\newtheorem{lemma}[theorem]{Lemma}

\newtheorem{proposition}[theorem]{Proposition}

\newtheorem{claim}{Claim}

\newcommand{\Continue}{\textbf{continue}}
\newcommand{\bdots}{..}

\newcommand\LPh{(LP$_h$)}
\newcommand\LPl{(LP$_\ell$)}
\newcommand\LPpl{(LP$_\ell^\prime$)}
\newcommand\Al{(LD$_\ell$)}

\newcommand\dd{{\mathrm d}}
\newcommand\rk{\mathrm{rk}}
\newcommand\NN{{\mathbb N}}
\newcommand\ZZ{{\mathbb Z}}
\newcommand\RR{{\mathbb R}}
\newcommand\TT{{\mathbb T}}

\begin{document}
\title{Curves on the torus with few intersections\thanks{
   The first, third and fourth authors have been supported by the MUNI Award in Science and Humanities (MUNI/I/1677/2018) of the Grant Agency of Masaryk University and by the project GA24-11098S of the Czech Science Foundation.
   The first and fourth authors also acknowledge support by the National Science Foundation under Grant No. DMS-1928930 while the authors were in residence at the Simons Laufer Mathematical Sciences Institute (formerly MSRI) in Berkeley, California, during the Spring 2025 semester.
   The work of the second author was supported by the OP JAK Project (MSCAFellow5\_MUNI (CZ.02.01.01/00/22\_010/0003229).
   The work presented in this manuscript appeared in the form of an extended abstract in the proceedings of Eurocomb'25.}}
\author{Igor Balla\thanks{Simons Laufer Mathematical Sciences Institute, 17 Gauss Way, Berkeley, CA 94720. E-mail: {\tt iballa1990@gmail.com}.}\and
        Marek Filakovsk\'y\thanks{Faculty of Informatics, Masaryk University, Botanick\'a 68A, 602 00 Brno, Czech Republic. E-mail: {\tt filakovsky@fi.muni.cz}.}\and
	Bart\l{}omiej Kielak\thanks{Institute of Mathematics, Leipzig University, Augustusplatz 10, 04109 Leipzig, Germany. Previous affiliation: Faculty of Informatics, Masaryk University, Botanick\'a 68A, 602 00 Brno, Czech Republic.}\and
        Daniel Kr{\'a}l'\thanks{Institute of Mathematics, Leipzig University, Augustusplatz 10, 04109 Leipzig, and Max Planck Institute for Mathematics in the Sciences, Inselstra{\ss}e 22, 04103 Leipzig, Germany. E-mail: {\tt daniel.kral@uni-leipzig.de}. Previous affiliation: Faculty of Informatics, Masaryk University, Botanick\'a 68A, 602 00 Brno, Czech Republic.}\and
	Niklas Schlomberg\thanks{Research Institute for Discrete Mathematics and Hausdorff Center for Mathematics, University of Bonn. E-mail: {\tt schlomberg@or.uni-bonn.de}}
	}

\date{}

\maketitle

\begin{abstract}
Aougab and Gaster [Math. Proc. Cambridge Philos. Soc. 174 (2023), 569--584]
proved that any set of simple closed curves on the torus,
where any two are non-homotopic and intersect at most $k$ times,
has a maximum size of $k+O(\sqrt{k}\log k)$.
We determine the maximum size of such a set for every $k$.
In particular, the maximum never exceeds $k+6$, and
it does not exceed $k+4$ when $k$ is large.

As this quantity coincides with the maximal number of columns of a generic $k$-modular matrix with two rows,
our result also settles the column number problem, a problem of interest in combinatorial optimization, for such matrices.
\end{abstract}

\section{Introduction}
\label{sec:intro}

We study families of simple closed homotopy-non-equivalent curves on a compact surface such that
any two have at most $k$ intersections. 
Despite its fundamental and elementary nature,
determining the maximum size of such families has remained an open problem, even in the case of a torus.
In this paper, we completely resolve the case of the torus,
i.e., we determine the maximum size of such a family for every $k\in\NN$.
Our result also completely settles the column number problem for generic matrices with two rows,
which is a problem concerning $\Delta$-modular matrices studied in combinatorial optimization in relation to integer programming.
This case of the column number problem
has recently been solved for sufficiently large $\Delta$ by Kriepke and Schymura~\cite{KriS25}, independently of our work;
we give further details in Subsection~\ref{subsec:IP} and also refer to the blog post by Ellenberg~\cite{Ell25-blog}.

Formally,
a \emph{$k$-system} on a surface $\Sigma$
is defined to be a collection of simple closed curves on $\Sigma$ such that
any two are non-homotopic and intersect at most $k$ times.
Determining the maximum size of a $k$-system on $\Sigma$,
which is denoted by $N(\Sigma, k)$,
has been the subject of an intensive line of research
for various surfaces $\Sigma$ and values $k$~\cite{Aou14,Aou17,Aou18,AouBG19,Gre18,Gre19,JuvMM96,MalRT14,Prz15},
also see~\cite{PacTT22} for results concerning a punctured plane.
We remark that, as discussed in~\cite{Gre19,Prz15},
particularly the case $k=1$ enjoys having interesting relations including those
to the systolic curve complex~\cite{Sch00} and Dehn surgery~\cite{BakGL15}.
A priori, it is not clear whether $N(\Sigma,k)$ is even finite;
to this end, Juvan, Malni\v c and Mohar~\cite{JuvMM96} showed that $N(\Sigma,k)$
is finite for every compact surface $\Sigma$ and every $k\in\NN$.
When $\Sigma$ is the closed orientable surface of genus $g$,
Greene~\cite{Gre19} showed that $N(\Sigma,k)\le O(g^{k+1}\log g)$ for any fixed $k\in\NN$,
improving Przytycki's bound from~\cite{Prz15}.
In this work, we focus on the case when $\Sigma$ is the torus $\TT^2$;
this arguably simplest case turns out to have surprising connections to number theory, which are presented below.

Juvan, Malni\v c and Mohar~\cite{JuvMM96} showed that $k + 1 \leq N(\TT^2, k) \leq 2k + 3$ and
noted that the upper bound can be improved to $\frac{3}{2}k + O(1)$.
The connection between this problem and number theory was pointed out by Agol~\cite{Ago00},
who observed that the size of a $k$-system on the torus is at most one more than the smallest prime larger than $k$.
This implies that $N(\TT^2,k)$ is at most $(1+o(1))k$ and specifically,
using the bound on the size of prime gaps by Baker, Harman and Pintz~\cite{BakHP01},
$N(\TT^2,k)$ is at most $k+O(k^{21/40})$.
Cram\'er~\cite{Cra21} showed that a positive resolution of the Riemann hypothesis
would yield a bound on prime gaps implying that $N(\TT^2, k)$ is at most $k+O(\sqrt{k}\log k)$ and
formulated a stronger number-theoretic conjecture that would imply an upper bound of $k+O(\log^2 k)$;
we refer to~\cite{Gra95,OliHP14} for further discussion
including the suspicion that Cram\'er's error term should actually be $O(\log^{2+\varepsilon} k)$.

Very recently,
Aougab and Gaster~\cite{AouG23} used combinatorial and geometric arguments
in conjunction with estimates from analytic number theory
to show that $N(\TT^2, k)$, i.e., the maximum size of a $k$-system on the torus,
is at most $k+O(\sqrt{k}\log k)$ (note that this matches the bound implied by Cram\'er's bound on prime gaps,
which is conditioned on a positive resolution of the Riemann hypothesis).

Our main result determines the maximum size of a $k$-system for every $k\in\NN$.
Aougab and Gaster also noted that
they are not aware of any $k$-system on the torus whose size exceeds $k+6$, and
our main result shows that there is indeed no $k$-system whose size exceeds $k+6$.

\begin{theorem}
\label{thm:main}
Let $K_0$ be the set containing the $59$ integers listed in Table~\ref{tab:main}.
For every $k\in\NN\setminus K _0$,
it holds that
\[
N(\TT^2,k) = \begin{cases}
k+4 &\text{ if } k \bmod 6 = 2,\\
k+3 &\text{ if } k \bmod 6\in\{1,3,5\}, \mbox{and}\\
k+2 &\text{ otherwise}.
\end{cases}
\]
The values of $N(\TT^2,k)$ for $k\in K_0$ are given in Table~\ref{tab:main}.
\end{theorem}

\noindent Note that there are only four values of $k$ such that $N(\TT^2,k)=k+6$, namely $k\in\{24,48,120,168\}$, and
only $13$ values such that $N(\TT^2,k)=k+5$.

\begin{table}
\begin{center}
\begin{tabular}{|c|cccccccccc|}
\hline
$k$ & 1 & 2 & 19 & 23 & 24 & 25 & 33 & 34 & 37 & 47 \\
\hline
$N(\TT^2,k)$ & 3 & 4 & 23 & 27 & 30 & 30 & 37 & 38 & 42 & 51 \\
$N(\TT^2,k)-k$ & +2 & +2 & +4 & +4 & +6 & +5 & +4 & +4 & +5 & +4 \\
``pattern'' & +3 & +4 & +3 & +3 & +2 & +3 & +3 & +2 & +3 & +3 \\
\hline
\hline
$k$ & 48 & 49 & 53 & 54 & 55 & 61 & 62 & 63 & 64 & 76 \\
\hline
$N(\TT^2,k)$ & 54 & 54 & 57 & 59 & 60 & 65 & 67 & 67 & 68 & 80 \\
$N(\TT^2,k)-k$ & +6 & +5 & +4 & +5 & +5 & +4 & +5 & +4 & +4 & +4 \\
``pattern'' & +2 & +3 & +3 & +2 & +3 & +3 & +4 & +3 & +2 & +2 \\
\hline
\hline
$k$ & 83 & 84 & 85 & 89 & 90 & 94 & 113 & 114 & 115 & 118 \\
\hline
$N(\TT^2,k)$ & 87 & 89 & 89 & 93 & 94 & 98 & 117 & 119 & 119 & 122 \\
$N(\TT^2,k)-k$ & +4 & +5 & +4 & +4 & +4 & +4 & +4 & +5 & +4 & +4 \\
``pattern'' & +3 & +2 & +3 & +3 & +2 & +2 & +3 & +2 & +3 & +2 \\
\hline
\hline
$k$ & 119 & 120 & 121 & 124 & 127 & 139 & 141 & 142 & 143 & 144 \\
\hline
$N(\TT^2,k)$ & 123 & 126 & 126 & 128 & 132 & 143 & 145 & 147 & 147 & 149 \\
$N(\TT^2,k)-k$ & +4 & +6 & +5 & +4 & +5 & +4 & +4 & +5 & +4 & +5 \\
``pattern'' & +3 & +2 & +3 & +2 & +3 & +3 & +3 & +2 & +3 & +2 \\
\hline
\hline
$k$ & 145 & 154 & 167 & 168 & 169 & 174 & 184 & 204 & 208 & 214 \\
\hline
$N(\TT^2,k)$ & 149 & 158 & 171 & 174 & 174 & 178 & 188 & 208 & 212 & 217 \\
$N(\TT^2,k)-k$ & +4 & +4 & +4 & +6 & +5 & +4 & +4 & +4 & +4 & +3 \\
``pattern'' & +3 & +2 & +3 & +2 & +3 & +2 & +2 & +2 & +2 & +2 \\
\hline
\hline
$k$ & 234 & 244 & 264 & 274 & 294 & 304 & 324 & 354 & 384 & \\
\hline
$N(\TT^2,k)$ & 238 & 247 & 268 & 277 & 297 & 307 & 327 & 357 & 387 & \\
$N(\TT^2,k)-k$ & +4 & +3 & +4 & +3 & +3 & +3 & +3 & +3 & +3 & \\
``pattern'' & +2 & +2 & +2 & +2 & +2 & +2 & +2 & +2 & +2 & \\
\hline
\end{tabular}
\end{center}
\caption{The values of $N(\TT^2,k)$ for $k\in K_0$.
         The values ``pattern'' are the additive constants based on $k\bmod 6$ given as in Lemma~\ref{lm:height123},
         which determines the values of $N(\TT^2,k)$ for all sufficiently large $k$.}
\label{tab:main}
\end{table}

Since the proof of Theorem~\ref{thm:main} is computer assisted,
we also give a weaker version of the theorem,
which is stated as Theorem~\ref{thm:main2}, that can be proven \emph{without computer assistance}.
In particular,
Theorem~\ref{thm:main2} asserts that $N(\TT^2,k)\le k+4$ for every sufficiently large $k\in\NN$,
which implies that $N(\TT^2,k)\le k+O(1)$ for all $k$.

\subsection{Column number problem}
\label{subsec:IP}

Integer programming is one of the most fundamental problems in combinatorial optimization.
An \emph{integer program} asks for determining the maximum value of $c^\top x$ subject to $Ax\le b$ and $x\in\ZZ^n$
where $A\in\ZZ^{m\times n}$, $b\in\ZZ^m$ and $c\in\ZZ^n$.
Integer programming is very hard from the computational complexity point of view:
it appeared among the 21 problems shown to be NP-complete in the original paper on NP-completeness by Karp~\cite{Kar72} and
it is known to remain NP-complete even when the entries of the constraint matrix $A$ are zero and one only.
A prominent tractable case is when the constraint matrix $A$ is totally unimodular,
i.e., all determinants of its square submatrices are equal to $0$ or $\pm 1$.
In this case, all vertices of the polyhedron $Ax\le b$, $x\in\RR^n$ are integral and
so linear programming algorithms can be applied.

An integer matrix $A$ is \emph{$\Delta$-modular}
if the absolute value of the determinant of every full-rank square matrix is at most $\Delta$;
we remark that such matrices are sometimes called $\Delta$-submodular
while the term $\Delta$-modular is reserved to be used for those matrices that
contain a full rank square submatrix with the determinant equal $\pm\Delta$.
It is conjectured that
integer programs with $\Delta$-modular constraint matrices can be efficiently solved~\cite{GriMPV18,She97}, and
efficient algorithms exist for integer programs with generic $\Delta$-modular constraint matrices~\cite{ArtEGOVW16},
see also~\cite{JiaB22};
a matrix $A$ is \emph{generic} if any $\rk(A)$ columns are linearly independent.
We also refer to~\cite{ArtWZ17,BonSEHN14,FioJWY22}
for additional results on integer programs with $\Delta$-modular constraint matrices.

The \emph{column number problem} asks for determining the maximum number of columns of a generic integer $\Delta$-modular matrix;
we refer to~\cite{ArtEGOVW16,ArtWZ17,AveS24,KriKS25,OxlW22,PaaSWX24} for results on this problem and
its variant when the matrix is assumed to be \emph{simple}, i.e., not containing a zero column or parallel columns, instead of generic.
In particular, Oxley and Walsh~\cite{OxlW22} determined the maximal number of columns of simple $2$-modular matrices.
Our work is directly related to the column number problem for matrices with two rows:
the vectors of a $\Delta$-nice set as defined in Subsection~\ref{subsec:overview}
form columns of a two-row integer matrix that is generic and $\Delta$-modular, and
vice versa, if an integer matrix with two rows is generic and $\Delta$-modular,
then its columns, each after dividing by the greatest common divisor of its entries (to make the entries coprime),
form a $\Delta$-nice set.
Since $k$-nice sets describe the homotopy classes of curves contained in a $k$-system,
we conclude that $N(\TT^2,\Delta)$
is equal to the maximal number of columns of a generic integer $\Delta$-modular matrix with two rows.
Kriepke and Schymura~\cite{KriS25} have recently determined the maximal number of columns of a generic integer $\Delta$-modular matrix 
when $\Delta\le 1\,550$ and $\Delta\ge 10^8$,
which determines $N(\TT^2,\Delta)$ for these values of $\Delta$;
their work is independent of our earlier work yielding analogous results~\cite{BalFKKS24v1}.
The main result of the current paper determines $N(\TT^2,\Delta)$ for all values of $\Delta$.

\subsection{Overview of the proof}
\label{subsec:overview}

We now provide a brief overview of the proof of Theorem~\ref{thm:main}, and
the proof of the upper bound of $k+O(1)$.
We believe the proof to be quite accessible as
it only utilizes widely known tools from combinatorics, discrete optimization, and geometry,
together with simple number-theoretic observations.

The proof is based on the analysis of the geometric and number-theoretic structure of a subset of $\ZZ^2$ that
describes the homotopy classes of curves contained in a $k$-system;
we introduce the relevant notation and basic properties in Section~\ref{sec:prelim}.
Those subsets of $\ZZ^2$ corresponding to $k$-systems on the torus will be called \emph{$k$-nice} throughout the paper.
In Section~\ref{sec:large},
we show that the area of the convex hull of any $k$-nice set is at most $\frac{\pi}{2}k$ and
show that the size of any $k$-nice set with height $h$ is at most $Ck+O(h)$ for a constant $C\in (0,1)$ (Lemma~\ref{lm:density});
the \emph{height} of a subset is the smallest $h$ such that it is contained in a strip with width $2h$ centered around the $x$-axis.
This result relies, implicitly, on the fact that the density of coprime points in $\ZZ^2$ is asymptotically equal to the Euler product $\prod_p\left(1-p^{-2}\right)$ over all primes $p$, which is equal to $\frac{1}{\zeta(2)}=\frac{6}{\pi^2}$ and, crucially, is strictly smaller than $\frac{2}{\pi}$.

In Section~\ref{sec:bound},
we focus on analyzing the height of $k$-nice sets and the size of $k$-nice sets with specific height.
First, we show that any $k$-nice set can be modified to a $k$-nice set of the same size
with height at most $\sqrt{2k}$ (Lemma~\ref{lm:sqroot}).
The size of a $k$-nice set with height $h$ is analyzed by a suitable linear program,
which yields that every $k$-nice set with height $h$
has at most $\gamma_h k+\beta_h$ elements for some $\gamma_h\in (0,1)$ (Lemmas~\ref{lm:size}~and~\ref{lm:LP}) whenever $h\ge 4$.
In particular,
if $k$ is sufficiently large, the size of $k$-nice set of height $h\ge 4$
is either at most $Ck+O(\sqrt{k})<k$ by Lemmas~\ref{lm:density} and~\ref{lm:sqroot} or
at most $\gamma_h k+\beta_h<k$ by Lemmas~\ref{lm:size}~and~\ref{lm:LP}.
This line of reasoning is refined using computer assistance to eventually yields that
for every $k\ge 1892$, there exists a $k$-nice set with maximum size that has height at most three (Theorem~\ref{thm:k0new}).

To complete the proof of Theorem~\ref{thm:main},
we determine the maximum size of $k$-nice sets with height at most three in Section~\ref{sec:three} and
we determine the maximum size of $k$-nice sets for $k\in\{3,\ldots,1891\}$ with computer assistance in Section~\ref{sec:algorithm}.
The claims used to prove Theorem~\ref{thm:main} that are computer assisted are listed in Table~\ref{tab:comp}.
Without computer assistance, it is possible to show a weaker version of Theorem~\ref{thm:k0new}
where $1892$ is replaced with a significantly larger constant $k_0$ (the value $k_0$ is set in the proof of Theorem~\ref{thm:main2});
this yields a computer-free proof that every $k$-nice set for $k\ge k_0$ has size at most $k+4$ and
so the general upper bound $k+O(1)$ on the size of a $k$-nice set (see the discussion before Theorem~\ref{thm:main2}).

\begin{table}
\begin{center}
{\renewcommand{\arraystretch}{1.8}
\begin{tabular}{|l|p{12.5cm}|}
\hline
Proof of & Computer assisted claim \\
\hline
Lemma~\ref{lm:density} &
  $\sum\limits_{i=1}^{\ell} \rho_{p_i}+\alpha_{p_i} \le \frac{4946}{3675}\cdot\ell$ for every $\ell\in\{1,\ldots,210\}$.\\
Theorem~\ref{thm:k0} &
  $\gamma_h\cdot\frac{h^2}{2}+\beta_h<\frac{h^2}{2}+3$ for every $h\in\{81,\ldots,41019\}$.\\
Theorem~\ref{thm:k0} &
  $3225\gamma_h+\beta_h<3225+3$ for every $h\in\{4,\ldots,80\}$.\\
Lemma~\ref{lm:sqroot-alg} &
  Algorithm~\ref{alg:height} returns \emph{verified} for all $k\in\{2,\ldots,3224\}$ and $\sqrt{4k/3}<h\le\sqrt{2k}$.\\
Theorem~\ref{thm:k0new} &
  $\gamma_h\cdot\frac{3h^2}{4}+\beta_h<\frac{3h^2}{4}+3$ for every $h\in\{51,\ldots,66\}$.\\
Theorem~\ref{thm:k0new} &
  $1892\gamma_h+\beta_h<1892+3$ for every $h\in\{4,\ldots,50\}$.\\
Theorem~\ref{thm:main_compute} &
  If Algorithm~\ref{alg:height} returns \emph{not verified}
  for $k\in\{3,\ldots,1891\}\setminus K_0$ and $h\in\{2,\ldots,\lfloor\sqrt{4k/3}\rfloor\}$,
  then \texttt{compute(k,h,N)} returns $N$.\\
Theorem~\ref{thm:main_compute} &
  For every $k\in K_0$,
  $N(\TT^2,k)$ given in Table~\ref{tab:main}
  is the maximum value returned by \texttt{compute(k,h,N)} called
  for all $h\in\{2,\ldots,\lfloor\sqrt{4k/3}\rfloor\}$ such that Algorithm~\ref{alg:height} returns \emph{not verified}.\\
\hline
\end{tabular}
}
\end{center}
\caption{The list of all claims used to prove Theorem~\ref{thm:main} that have been verified with computer assistance.
         The files containing the code needed to verify the claims
	 are available as ancillary files with the arXiv version of the manuscript.}
\label{tab:comp}
\end{table}

\section{Preliminaries}
\label{sec:prelim}

We now recall basic results concerning closed simple curves in the torus, which we use later;
we refer to e.g.~\cite{Sti12} for a more detailed exposition.
We view the torus as $\RR^2/\ZZ^2$ and
let $C_{m,n}$ for $(m,n)\in\ZZ^2$, $(m,n)\not=0$, be the closed curve in the torus parameterized as $(m\cdot t\mod 1,n\cdot t\mod 1)$ for $t\in[0,1]$.
Every non-trivial\footnote{A closed curve is trivial if it is homotopic to a point.} closed curve in the torus is freely homotopic to $C_{m,n}$ for some non-zero $(m,n)\in\ZZ^2$;
if the curve is non-self-intersecting, then $m$ and $n$ are coprime, i.e., $\gcd(m,n)=1$.
Note that $(2,0)$ is not a coprime pair as $\gcd(2,0)=2$.

Consider coprime pairs $(m,n)$ and $(m',n')$ of non-zero integers.
The minimum number of crossings of closed simple curves freely homotopic to $C_{m,n}$ and $C_{m',n'}$ is equal to $|mn'-m'n|$, and
this minimum is attained by the curves $C_{m,n}$ and $C_{m',n'}$ themselves.
This leads us to the following definition:
for an integer $k\in\NN$,
we say that a set $Q\subseteq\ZZ^2$ is \emph{$k$-nice} if
\begin{itemize}
\item $Q$ contains non-zero coprime pairs only,
\item it does not contain both $(m,n)$ and $(-m,-n)$ for any $(m,n)\in\ZZ^2$, and
\item $|mn'-m'n|\le k$ for all $(m,n)$ and $(m',n')$ contained in $Q$.
\end{itemize}
Since $(m,n)$ and $(-m,-n)$ represent opposite orientations of the same curve,
any $k$-system of simple closed curves on the torus can be represented by a $k$-nice set $Q\subseteq\ZZ^2$, and
vice versa.
Hence, to prove Theorem~\ref{thm:main}, it suffices to determine the maximum size of a $k$-nice set for every $k\in\NN$.

Let $A \in \ZZ^{2\times 2}$ be a unimodular matrix, i.e., $A$ has integer coordinates and $|\det A|=1$.
Observe that
if $Q\subseteq\ZZ^2$ is $k$-nice, then the set
\[AQ=\{Ax, x\in Q\}\]
is also $k$-nice.
Indeed, since $A$ is unimodular, the inverse $A^{-1}$ is also unimodular and so,
if $Ax$ were not coprime, i.e., both coordinates of $Ax$ were divisible by an integer $s>1$,
then both coordinates of $A^{-1}Ax=x$ would also be divisible by $s$.
Likewise, if $x\in Q$ and $y\in Q$,
then $|\det (x|y)|=| x_1y_2-x_2y_1|\le k$ and so
\[|(Ax)_1(Ay)_2-(Ax)_2(Ay)_1|=|\det (Ax|Ay)|=|\det A|\cdot|\det (x|y)|=|\det (x|y)|\le k.\]
Note that the application of the matrix $A$ to the elements of the set $Q\subseteq\ZZ^2$
corresponds to a reparameterization of the torus.
Also note that negating both coordinates of an element of $Q$ does not change what curve it corresponds to.
Hence, we say that two subsets $Q$ and $Q'$ of $\ZZ^2$ are \emph{equivalent}
if there exists an integer unimodular matrix $A$ such that $Q'$ can be obtained from $AQ$ by negating a subset of its elements.
Since the operations of multiplying each element by $A$ and negating a subset of elements commute, and
$A^{-1}$ is also unimodular,
it is not hard to see that being equivalent is indeed an equivalence relation on $k$-nice sets.

We say that $Q\subseteq\ZZ^2$ is
\emph{$x$-non-negative} if $m\ge 0$ for all $(m,n)\in Q$, and
it is
\emph{$y$-non-negative} if $n\ge 0$ for all $(m,n)\in Q$.
The \emph{height} of a set $Q\subseteq\ZZ^2$ is the maximum $z\in\NN$ such that the set $Q$ contains $(m,n)$ with $|n|=z$, and
the \emph{width} is the maximum $z\in\NN$ such that the set $Q$ contains $(m,n)$ with $|m|=z$.

We conclude this section with stating a closure property of inclusion-wise maximal $k$-nice $y$-non-negative sets.
The property is equivalent to Lemma~2.9 in~\cite{AouG23}, which is attributed to Agol in~\cite{AouG23}.
We include a short proof for completeness.

\begin{proposition}
\label{prop:maximal}
Let $Q$ be an inclusion-wise maximal $k$-nice $y$-non-negative subset of $\ZZ^2$.
If $(m,n)\in\ZZ^2$ is contained in the convex hull of $Q$ and the integers $m$ and $n$ are coprime,
then $(m,n)$ is contained in $Q$.
\end{proposition}

\begin{proof}
Fix an inclusion-wise maximal $k$-nice $y$-non-negative set $Q\subseteq\ZZ^2$ and
let $(m,n)\in\ZZ^2$ be contained in the convex hull of $Q$ such that $m$ and $n$ are coprime.
Note that $n\geq 0$ as $Q$ is $y$-non-negative.

Suppose that a point $(m,n)\in\ZZ^2$ with $m$ and $n$ being coprime is not contained in~$Q$;
note that $(-m,-n)$ is also not contained in $Q$ as $Q$ is $y$-non-negative.
The inclusion-wise maximality of $Q$ implies that
there exists $(m',n')\in Q$ such that $|mn'-m'n|>k$.
If $mn'-m'n>k$,
then there must exist $(m'',n'')\in Q$ such that $m''n'-m'n''>k$ (as $(m,n)$
is contained in the convex hull of $Q$ and any linear function on a convex set is maximized at a boundary point),
which is impossible.
Similarly,
if $mn'-m'n<-k$, then there must exist $(m'',n'')\in Q$ such that $m''n'-m'n''<-k$, which is also impossible.
We conclude that $(m,n)$ is contained in~$Q$.
\end{proof}

\section{Nice sets with large height}
\label{sec:large}

In this section, we show that, when $h$ is sufficiently large,
the size of a $k$-nice set with height $h$ does not exceed $\gamma k+O(h)$ for some $\gamma\in (0,1)$.
We start with the following auxiliary geometric result,
which is a geometric analogue of the upper bound that we wish to prove.

\begin{lemma}
\label{lm:area}
Let $k\in\NN$.
If $S\subseteq\RR^2$ is a convex set such that $y\ge 0$ for every $(x,y)\in S$ and
$|xy'-yx'|\le k$ for all $(x,y),(x',y')\in S$,
then the area of $S$ is at most $\pi k/2$.
\end{lemma}

\begin{proof}
It is enough to prove the lemma for $k=1$.
Indeed, if $k>1$,
the set $S'=\left\{\left(x/\sqrt{k}, y/\sqrt{k}\right)  :  (x,y)\in S\right\}$ satisfies
the assumption of the lemma for $k=1$ and its area is the area of $S$ divided by $k$.

Fix a convex set $S\subseteq\RR^2$ such that $y\ge 0$ for every $(x,y)\in S$ and
$|xy'-yx'|\le 1$ for all $(x,y),(x',y')\in S$.
We may assume that $S$ is closed since the closure of $S$ also satisfies the assumption of the lemma.
In addition, we may assume that $S$ is bounded:
if the area of any convex bounded subset of $S$ is at most $\pi/2$,
then the area of $S$ is at most $\pi/2$.
Finally, we may assume that $\max\{y  :  (x,y)\in S\}=1$ (note that maximum exists as $S$ is bounded and closed):
if $\max\{y  :  (x,y)\in S\}=z\not=1$,
then we can consider $S'=\left\{\left( xz,y/z\right)  :  (x,y)\in S\right\}$ instead of $S$,
which has the same area as $S$ and also satisfies the assumptions of the lemma.

We next define two auxiliary functions $f^-,f^+:[0,1]\to\RR$ as follows:
\begin{align*}
f^-(y) & =\min\{x  :  (x,y)\in S\},\\
f^+(y) & =\max\{x  :  (x,y)\in S\}.
\end{align*}
For $y\in [0,1]$, let $y'=\sqrt{1-y^2}$;
since the points $(f^+(y),y)$ and $(f^-(y'),y')$ are contained in $S$,
we get
\[f^+(y)-\frac{y}{\sqrt{1-y^2}}f^-\left(\sqrt{1-y^2}\right)\le\frac{1}{\sqrt{1-y^2}}.\]
It follows that
\begin{equation}
\int\limits_{0}^{1}f^+(y)\dd y-\int\limits_{0}^{1}\frac{y}{\sqrt{1-y^2}}f^-\left(\sqrt{1-y^2}\right)\dd y
\le \int\limits_{0}^{1}\frac{1}{\sqrt{1-y^2}}\dd y=\frac{\pi}{2}.
\label{eq:area1}
\end{equation}
We next obtain by substituting that
\begin{equation}
\int\limits_{0}^{1}\frac{y}{\sqrt{1-y^2}}f^-\left(\sqrt{1-y^2}\right)\dd y=
-1 \int\limits_{1}^{0}f^-(y)\dd y=\int\limits_{0}^{1}f^-(y)\dd y.
\label{eq:area2}
\end{equation}
We combine \eqref{eq:area1} and \eqref{eq:area2} to conclude that the area of $S$ is
\[
\int\limits_{0}^{1}f^+(y)-f^-(y)\dd y=
\int\limits_{0}^{1}f^+(y)\dd y-\int\limits_{0}^{1}\frac{y}{\sqrt{1-y^2}}f^-\left(\sqrt{1-y^2}\right)\dd y \le \frac{\pi}{2},\]
which finishes the proof.
\end{proof}

We next define quantities $\rho_{\ell}$ and $\alpha_{\ell}$ for every $\ell\in\NN$;
their numerical values for $\ell\in\{1,\ldots,20\}$ can be found in Table~\ref{tab:rhobeta} in Section~\ref{sec:bound}.
A python script \verb|compute_rho_alpha_beta_gamma.py| for computing the values and
another script that verifies the inequalities claimed in the proofs of Lemma~\ref{lm:density},
Theorem~\ref{thm:k0} and Theorem~\ref{thm:k0new} (see Table~\ref{tab:comp})
are available as ancillary files with the arXiv version of the manuscript.
The quantities are defined as follows:
\begin{align*}
\rho_{\ell} & =\prod_{\mbox{{\scriptsize primes} } p,\, p|\ell}\left(1-\frac{1}{p}\right)\mbox{ and}\\
\alpha_{\ell} & =\max_{1\le a\le b\le 2\ell}\left|\{z, a\le z\le b\mbox{ and }\gcd(z,\ell)=1\}\right|-\rho_{\ell} (b-a+1).
\end{align*}
Observe that if $X$ is a set of $\ell$ consecutive integers,
then $X$ contains exactly $\rho_{\ell}\ell$ elements that are coprime with $\ell$, and
so if $X$ is a set of $n=k\ell+k'$, $0\le k'<\ell$, consecutive integers,
then at most
\[\rho_{\ell}k\ell+\rho_{\ell}k'+\alpha_\ell=\rho_{\ell} n+\alpha_{\ell}\]
elements of $X$ are coprime with $\ell$.
We state this fact as a proposition for future reference.

\begin{proposition}
\label{prop:rhoalpha}
Let $X$ be a set of $n$ consecutive integers and let $\ell\in\NN$.
The number of $x\in X$ that are coprime with $\ell$ is at most $\rho_{\ell} n+\alpha_{\ell}$.
\end{proposition}

We are now ready to prove the main lemma of this section.

\begin{lemma}
\label{lm:density}
For every $h\ge 41020$ and every $k\ge h$,
the maximum size of a $k$-nice set with height $h$ is at most
\[\frac{3264\pi}{10255}\cdot k+\frac{4946}{3675}\cdot h+1.\]
\end{lemma}

\begin{proof}
Set $h_0=41020$.
Consider a $k$-nice set $Q$ with height $h\ge h_0$ for some $k\ge h$;
without loss of generality, we may assume that $(1,0)\in Q$ (we use the maximality of $Q$ and that $h\le k$) and
$Q$ is $y$-non-negative (by replacing any element $(x,y)$ with $y<0$ with the element $(-x,-y)$).
Let $\widehat{Q}$ denote the convex hull of $Q$,
let $s_i$ and $t_i$ be the minimum and maximum real such that
the points $(s_i,i)$ and $(t_i,i)$ are contained in $\widehat{Q}$ for $i=0,\ldots,h$, and
let $\ell_i=t_i-s_i$.
Note that $\ell_0=0$.
Finally, for $i=1,\ldots,h$,
let $P_i$ be the set of all $p\in\{2,3,5,7\}$ such that $p|i$ (for instance, $P_{132}=\{2,3\}$) and let
$p_i$ be the product of the elements contained in $P_i$;
if $P_i=\emptyset$, we set $p_i=1$.

Since $\widehat{Q}$ is convex, the sequence $\ell_0,\ldots,\ell_h$ is concave and in particular unimodal,
i.e., the values of $\ell_i$'s first increase and then decrease;
let $m\in\{0,\ldots,h\}$ be such such that $\ell_m$ is the maximum element of this sequence.
Lemma~\ref{lm:area} and the convexity of~$\widehat{Q}$ imply that
\begin{equation}
\sum_{i=1}^{h}\ell_i
  =\frac{\ell_h}{2}+\sum_{i=1}^{h}\frac{\ell_{i-1}+\ell_i}{2}
  \le\frac{\pi}{2}\cdot k+\frac{\ell_m}{2},
\label{eq:sumell}
\end{equation}
as the sum in the middle expression is a lower bound on the area of $\widehat{Q}$.
On the other hand, note that $\widehat{Q}$ contains the triangles with corners $( s_m, m), (t_m, m), (s_0, 0)$ and $(s_m, m), (t_m, m), (s_h, h)$, which have a combined area of $h\ell_m/2$. Thus, the area of $\widehat{Q}$ is at least $h\ell_m/2$, which yields using Lemma~\ref{lm:area} that
\begin{equation}
\ell_m\le\frac{2}{h}\cdot\frac{\pi}{2}\cdot k\le\frac{\pi}{h_0}\cdot k.
\label{eq:ellm}
\end{equation}

Let $H$ be the smallest multiple of $210=2\cdot 3\cdot 5\cdot 7$ larger than $h$ and set $\ell_{h+1}=\cdots=\ell_H=0$;
note that the sequence $\ell_0,\ldots,\ell_H$ is unimodal.
Define $I_a$ for $a\in\{1,\ldots,210\}$ as the set of all $i\in\{1,\ldots,H\}$ such that $i=a\mod 210$;
note that $\lvert I_a\rvert=H/210$.
We next show that the following holds for any $a,b\in\{1,\ldots,210\}$:
\begin{equation}
\left\lvert\sum_{i\in I_a}\ell_i-\sum_{i\in I_b}\ell_i\right\rvert\le\ell_m.
\label{eq:ab}
\end{equation}
By symmetry, we may assume that $a<b$ (the inequality \eqref{eq:ab} trivially holds if $a=b$).
Observe that there exists $A$ such that $\ell_{210A+a}\ge\ell_i$ for all $i\in I_a\cup I_b$ or
there exists $B$ such that $\ell_{210B+b}\ge\ell_i$ for all $i\in I_a\cup I_b$.
The two cases are completely analogous and so we analyze the former case only.
Observe that $\ell_{210j+a}\le\ell_{210j+b}$ for $j\in\{0,\ldots,A-1\}$ and
$\ell_{210(j+1)+a}\le\ell_{210j+b}$ for $j\in\{A,\ldots,H/210-2\}$.
We obtain that
\[\sum_{i\in I_a}\ell_i
  \le\ell_{210A+a}+\sum_{j=0,j\not=A}^{H/210-1}\ell_{210j+a}
  \le\ell_{210A+a}+\sum_{j=0}^{H/210-2}\ell_{210j+b}
  \le\ell_m+\sum_{i\in I_b}\ell_i.\]
Likewise, it holds that $\ell_{210j+b}\le\ell_{210(j+1)+a}$ for $j\in\{0,\ldots,A-1\}$ and
$\ell_{210j+b}\le\ell_{210j+a}$ for $j\in\{A,\ldots,H/210-1\}$, and
we obtain that
\[\sum_{i\in I_b}\ell_i=\sum_{j=0}^{H/210-1}\ell_{210j+b}
  \le\sum_{j=1}^{A}\ell_{210j+a}+\sum_{j=A}^{H/210-1}\ell_{210j+a}
  \le\ell_{210A+a}+\sum_{i\in I_a}\ell_i
  \le\ell_{m}+\sum_{i\in I_a}\ell_i.\]
Hence, the inequality \eqref{eq:ab} follows.
Using \eqref{eq:sumell} and \eqref{eq:ab}, we obtain that the following holds for every $a\in\{1,\ldots,210\}$:
\[
\sum_{i\in I_a}\ell_i
 \le\ell_m+\frac{1}{210}\sum_{i=1}^{H}\ell_i
 \le\frac{\pi}{2\cdot 210}\cdot k+\frac{3\ell_m}{2},
\]
which yields using \eqref{eq:ellm} that
\begin{equation}
\sum_{i\in I_a}\ell_i
 \le\frac{\pi}{420}\cdot k+\frac{3\pi}{2h_0}\cdot k
 =\frac{\pi}{2}\cdot\left(\frac{1}{210}+\frac{3}{h_0}\right)\cdot k
\label{eq:Ia}
\end{equation}
Finally, we obtain using \eqref{eq:Ia} the following:
\begin{align}
\sum_{i=1}^h \ell_i\prod_{p\in P_i}\left(1-\frac{1}{p}\right)
& = \sum_{a=1}^{210}\left(\sum_{i\in I_a}\ell_i\right)\prod_{p\in P_a}\left(1-\frac{1}{p}\right)\nonumber\\
& \le \frac{\pi}{2}\cdot\left(\frac{1}{210}+\frac{3}{h_0}\right)\cdot k\cdot\sum_{a=1}^{210}\prod_{p\in P_a}\left(1-\frac{1}{p}\right)\nonumber\\
& = \frac{\pi}{2}\cdot\left(\frac{1}{210}+\frac{3}{h_0}\right)\cdot k\cdot 210\cdot\prod_{p=2,3,5,7}\left(1-\frac{1}{p^2}\right)
  = \frac{3264\pi}{10255}\cdot k. \label{eq:density}
\end{align}
For $i=1,\ldots,h$, let $Q_i$ be the set of the elements of $Q$ with their second coordinate equal to $i$.
Proposition~\ref{prop:rhoalpha} yields that
\begin{equation}
|Q_i|\le \rho_{p_i}(\ell_i+1)+\alpha_{p_i}=\rho_{p_i}+\alpha_{p_i}+\ell_i\prod_{p\in P_i}\left(1-\frac{1}{p}\right)
\label{eq:Qi}
\end{equation}
since at most $\rho_{p_i}(\ell_i+1)+\alpha_{p_i}$ integers among $\lceil s_i\rceil,\ldots,\lfloor t_i\rfloor$
are coprime with $p_i$ (note that every integer coprime with $i$ is also coprime with $p_i$).
Using computer assistance,
we have verified for all $\ell\in\{1,\ldots,210\}$ (the equality is attained for $\ell=210$) that
\begin{equation}
\sum_{i=1}^{\ell} \rho_{p_i}+\alpha_{p_i}
  \le \frac{4946}{3675}\cdot\ell.
\label{eq:sum210}
\end{equation}
We now combine \eqref{eq:Qi} with \eqref{eq:density} and \eqref{eq:sum210} to get that
\[\sum_{i=1}^h |Q_i|\le \frac{3264\pi}{10255}\cdot k+
  \sum_{i=1}^{210}\left\lceil\frac{h-i+1}{210}\right\rceil\cdot\left(\rho_{p_i}+\alpha_{p_i}\right)
  \le \frac{3264\pi}{10255}\cdot k+\frac{4946}{3675}\cdot h,\]
where the first inequality holds
since the number of $j\in [h]$ such that $j=i\mod 210$
is exactly $\left\lceil\frac{h-i+1}{210}\right\rceil$ for every $i\in [210]$.
Since $(1,0)$ is the only element of $Q$ with a zero $y$-coordinate, the statement of the lemma follows.
\end{proof}

\section{Bounding the height of a nice set}
\label{sec:bound}

In this section, we show that if $k$ is sufficiently large,
we may assume that the height of a maximum size $k$-nice set is at most three.
We start with showing a sublinear upper bound.

\begin{lemma}
\label{lm:sqroot}
For any $k$-nice set $Q$,
there exists a $k$-nice set equivalent to $Q$ that has height at most $\sqrt{2k}$.
\end{lemma}

\begin{proof}
Consider a $k$-nice set $Q$ and
let $Q_0$ be a set equivalent to $Q$ with the smallest possible height $h$;
by negating points if needed, we may assume that $Q_0$ is $y$-non-negative.

If $h\le\sqrt{2k}$, we are done.
Suppose for contradiction that $h>\sqrt{2k}$, and
let $(x_0,h)$ be the point in $Q_0$ with the maximum second coordinate.
By considering the set $A^mQ_0$ for a suitable $m\in\ZZ$,
where $A$ is the matrix
\[A=\begin{bmatrix} 1 & 1 \\ 0 & 1 \end{bmatrix},\]
we can assume that $|x_0|\le h/2$.
Observe that if $(x,y)\in Q_0$, then
\[\left|x-\frac{x_0}{h}y\right|\le\frac{k}{h},\]
which yields 
\[|x|\le\frac{|x_0|}{h}y+\frac{k}{h}\le |x_0|+\frac{k}{h}\le\frac{h}{2}+\frac{k}{h}.\]
We conclude that the width of $Q_0$ is at most $\frac{h}{2}+\frac{k}{h}$.
In particular, the height of the set $A'Q$,
where $A'$ is the matrix
\[A'=\begin{bmatrix} 0 & 1 \\ 1 & 0 \end{bmatrix},\]
is at most
\[\frac{h}{2}+\frac{k}{h}<\frac{h}{2}+\frac{\sqrt{2k}}{2}<h,\]
which contradicts the choice of $Q_0$ as a set equivalent to $Q$ that has the smallest possible height.
\end{proof}

We next consider the following linear program \LPl{}
with $2\ell$ variables $\sigma_1,\ldots,\sigma_{\ell}$ and $\tau_1,\ldots,\tau_{\ell}$ defined as follows:
\begin{align*}
& \mbox{maximize } & \sum_{i=1}^\ell \rho_i (\tau_i-\sigma_i) \\
& \mbox{subject to }
 & \qquad \tau_i \ge \sigma_i & \ge 0 & & \mbox{for all $1\le i\le \ell$, and}\\
 & & -1\le i\tau_j-j\sigma_i & \le 1 & & \mbox{for all $1\le i,j\le \ell$.}
\end{align*}
The objective value of the program \LPl{} is denoted by $\gamma_{\ell}$ for $\ell\in\NN$;
again, the values of $\gamma_{\ell}$ for $\ell\in\{1,\ldots,20\}$ can be found in Table~\ref{tab:rhobeta};
we will show that $\gamma_{\ell}<1$ for every $\ell\ge 4$ (see Lemma~\ref{lm:LP}).
Finally, define $\beta_0=1$ and $\beta_{\ell}=\beta_{\ell-1}+\alpha_{\ell}+\rho_{\ell}$;
again, the values of $\beta_{\ell}$ for $\ell\in\{1,\ldots,20\}$ are in Table~\ref{tab:rhobeta}.

\begin{table}
\begin{center}
\begin{tabular}{|c|cccc|}
\hline
$\ell$ & $\rho_{\ell}$ & $\alpha_{\ell}$ & $\gamma_{\ell}$ & $\beta_{\ell}$ \\
\hline
1 & 1.0000 & 0.0000 & 1.0000 &  2.0000 \\
2 & 0.5000 & 0.5000 & 1.0000 &  3.0000 \\
3 & 0.6667 & 0.6667 & 1.0000 &  4.3333 \\
4 & 0.5000 & 0.5000 & 0.9722 &  5.3333 \\
5 & 0.8000 & 0.8000 & 0.9917 &  6.9333 \\
6 & 0.3333 & 1.0000 & 0.9667 &  8.2667 \\
7 & 0.8571 & 0.8571 & 0.9752 &  9.9810 \\
8 & 0.5000 & 0.5000 & 0.9687 & 10.9810 \\
9 & 0.6667 & 0.6667 & 0.9695 & 12.3143 \\
10 & 0.4000 & 1.2000 & 0.9586 & 13.9143 \\
11 & 0.9091 & 0.9091 & 0.9679 & 15.7325 \\
12 & 0.3333 & 1.0000 & 0.9601 & 17.0658 \\
13 & 0.9231 & 0.9231 & 0.9680 & 18.9120 \\
14 & 0.4286 & 1.2857 & 0.9645 & 20.6262 \\
15 & 0.5333 & 1.3333 & 0.9605 & 22.4929 \\
16 & 0.5000 & 0.5000 & 0.9553 & 23.4929 \\
17 & 0.9412 & 0.9412 & 0.9617 & 25.3753 \\
18 & 0.3333 & 1.0000 & 0.9576 & 26.7086 \\
19 & 0.9474 & 0.9474 & 0.9634 & 28.6033 \\
20 & 0.4000 & 1.2000 & 0.9615 & 30.2033 \\
\hline
\end{tabular}
\end{center}
\caption{The numerical values of $\rho_{\ell}$, $\alpha_{\ell}$, $\gamma_{\ell}$ and $\beta_{\ell}$ for $\ell\in\{1,\ldots,20\}$.}
\label{tab:rhobeta}
\end{table}

We next relate the linear program \LPl{} to the sizes of $k$-nice sets.

\begin{lemma}
\label{lm:size}
Let $k\in\NN$.
For every $h\in\{1,\ldots,k\}$,
every $k$-nice set $Q\subseteq\ZZ^2$ with height exactly $h$ has at most $\gamma_h k+\beta_h$ elements.
\end{lemma}

\begin{proof}
Fix $k\in\NN$ and $h\in\{1,\ldots,k\}$,
Let $Q$ be a $k$-nice set with height exactly $h$.
By negating some of the elements of $Q$, we may assume that $Q$ is $y$-non-negative, and
by considering the set $A^m Q$ for sufficiently large $m\in\NN$ (if needed),
where
\[A=\begin{bmatrix} 1 & 1 \\ 0 & 1 \end{bmatrix},\]
we can also assume that the set $Q$ is $x$-non-negative.
Finally, since the height of $Q$ is $h\le k$, we can assume that $(1,0)\in Q$.

Let $s_i$ and $t_i$ for $i=1,\ldots,h$ be the minimum and maximum reals such that
$(s_i,i)$ and $(t_i,i)$ are in the convex hull of $Q$ (note that the values are well-defined
since $(1,0)\in Q$ and the height of $Q$ is exactly $h$).
Since any two points $(x,y)$ and $(x',y')$ of $Q$ satisfy $|xy'-x'y|\le k$,
it also holds that $|xy'-x'y|\le k$ for any two points $(x,y)$ and $(x',y')$ of the convex hull of $Q$.
Since $\sigma_i=s_i/k$ and $\tau_i=t_i/k$, $i=1,\ldots,h$, is a feasible solution of \LPh{},
we obtain that
\begin{equation}
\sum_{i=1}^h\rho_i (t_i-s_i)=k\sum_{i=1}^h\rho_i (\tau_i-\sigma_i)\le \gamma_h k.
\label{eq:lm:size}
\end{equation}

For $i\in\{0,\ldots,h\}$, let $Q_i$ be the points contained in $Q$ with their second coordinate equal to $i$.
Note that $|Q_0|=1$ and Proposition~\ref{prop:rhoalpha} yields that
\[|Q_i|\le \rho_i(t_i-s_i+1)+\alpha_i=\rho_i(t_i-s_i)+(\alpha_i+\rho_i)\]
for every $i\in\{1,\ldots,h\}$.
It follows using \eqref{eq:lm:size} and the definition of $\beta_h$ that
\[|Q|=\sum_{i=0}^h|Q_i|=\left(\sum_{i=1}^h \rho_i(t_i-s_i)\right)+1+\sum_{i=1}^h(\alpha_i+\rho_i)\le \gamma_h k+\beta_h.\]
This concludes the proof of the lemma.
\end{proof}

\subsection{Analysis of the linear program}

We now show that the optimal value of the linear program \LPl{} is smaller than $1$ for every $\ell\ge 4$.
We use a relaxed version, which is denoted by \LPpl{}, to analyze its optimum value.
The linear program \LPpl{} for $\ell\in\NN$ also has $2\ell$ variables $\sigma_1,\ldots,\sigma_\ell$ and $\tau_1,\ldots,\tau_\ell$ and
is defined as follows (note that the objective function is the same):
\begin{align*}
& \mbox{maximize } & \sum_{i=1}^\ell \rho_i (\tau_i-\sigma_i) \\
& \mbox{subject to }
 & \quad \tau_i, \sigma_i & \ge 0 & & \mbox{for all $1\le i\le \ell$, and}\\
 & & \frac{\tau_j}{j}-\frac{\sigma_i}{i} & \le \frac{1}{ij} & & \mbox{for all $1\le i,j\le \ell$.}
\end{align*}
Since any feasible solution of  \LPl{} is also a feasible solution of \LPpl{},
we obtain the following.

\begin{lemma}
\label{lm:LPp}
For every $\ell\in\NN$,
the optimum value of \LPl{} is at most the optimum value of \LPpl{}.
\end{lemma}

We are now ready to provide an upper bound on the optimum value of the linear program \LPl{},
which is obtained using the linear program \LPpl{}.

\begin{lemma}
\label{lm:LP}
Let $\ell\in\NN$.
If $\ell\in\{1,2,3\}$, then the optimum value of \LPl{} is equal to one, and
if $\ell\ge 4$, then the optimum value of \LPl{} is strictly smaller than one.
\end{lemma}

\begin{proof}
We first show that the optimum value of \LPl{} is at least one if $\ell\in\{1,2,3\}$.
Indeed, the following feasible solutions of \LPl{}
\begin{align*}
\ell=1:\quad & \sigma_1=0,\, \tau_1=1 \\
\ell=2:\quad & \sigma_1=0,\, \tau_1=3/4,\, \sigma_2=1/2,\, \tau_2=1 \\
\ell=3:\quad & \sigma_1=0,\, \tau_1=5/9,\, \sigma_2=1/3,\, \tau_2=7/9,\, \sigma_3=2/3,\, \tau_3=1
\end{align*}
have objective value equal to $1$.
Hence, it remains to show that the optimum value of \LPl{} is at most one if $\ell\in\{1,2,3\}$ and
strictly smaller than one if $\ell\ge 4$.

Fix $\ell\in\NN$ and consider the dual of the linear program \LPpl{},
which will be denoted by \Al{}:
\begin{align*}
& \mbox{minimize} & \sum_{i,j=1}^\ell \frac{1}{ij}A_{i,j} \\
& \mbox{subject to}
 & \quad A_{i,j} & \ge 0 & & \mbox{for all $1\le i,j\le \ell$,}\\
 & & \sum_{j=1}^\ell -\frac{1}{i}A_{i,j} & \le -\rho_i & & \mbox{for all $1\le i\le \ell$, and} \\
 & & \sum_{i=1}^\ell \frac{1}{j}A_{i,j} & \le \rho_j & & \mbox{for all $1\le j\le \ell$.}
\end{align*}
\noindent Recall that any feasible solution of the dual program \Al{} provides
an upper bound on the optimal value of the primal program \LPpl{} and
so the value of any feasible solution of \Al{} is an upper bound on the optimum value of \LPl{} by Lemma~\ref{lm:LPp}.

We now present the linear program \Al{} in a form suitable for our analysis.
Let $v_\ell\in\RR^\ell$ be the vector whose $i$-th coordinate is equal to $1/i$.
Also recall that Euler's totient function $\phi(k)$ is defined as follows:
$\phi(1)=1$ and $\phi(k)$ for $k\ge 2$      
is the number of positive integers smaller than $k$ that are coprime with $k$.
Observe that $\phi(k)=k\rho_k$ for every $k\in\NN$.
The optimization problem of \Al{} asks for minimizing $v_\ell^T Av_\ell$
over all \emph{non-negative} matrices $A\in\RR^{\ell\times \ell}$ such that
for each $i\in\{1,\ldots,\ell\}$,
the sum of the entries of the $i$-th row is bounded from below by Euler's totient function $\phi(i)$, and
the sum of entries of the $i$-th column is bounded from above by $\phi(i)$, i.e.,
\[\sum_{j=1}^\ell{}A_{i,j} \ge \phi(i) \qquad\mbox{and}\qquad \sum_{j=1}^\ell{}A_{j,i} \le \phi(i).\]
\noindent Observe that the entries of the matrix $A$ are the variables of the dual of \LPpl{},
where $A_{i,j}$ is the variable associated with the constraint for $i$ and $j$ in the definition of \LPpl{}.

We illustrate the presentation of \Al{} by giving examples of two feasible solutions for $\ell=4$:
\[A=\begin{bmatrix} 0 & 0 & 0 & 1 \\ 0 & 0 & 1 & 0 \\ 0 & 1 & 0 & 1 \\ 1 & 0 & 1 & 0 \end{bmatrix}
  \qquad\mbox{and}\qquad
  A=\begin{bmatrix} 0 & 0 & 0 & 1 \\ 0 & 0 & 0 & 1 \\ 0 & 0 & 2 & 0 \\ 1 & 1 & 0 & 0 \end{bmatrix}.\]
Observe that the definition of \Al{} yields that
the sum of the entries of $A$ in the $i$-th row is at least $\phi(i)$ and
the sum of the entries of $A$ in the $i$-th column is at most $\phi(i)$.
This implies the sum of all entries of $A$ is equal to $\phi(1)+\cdots+\phi(h)$ and
so all inequalities in the definition of \Al{} must hold with equality.
It follows that we could assume without loss of generality that the matrix $A$ in \Al{} is symmetric as
if $A$ is a feasible solution of \Al{}, then $(A+A^T)/2$ is also a feasible solution with the same objective value.

We now find a feasible solution of \Al{} with objective value equal to one,
which will imply that the optimum value of \LPl{} is at most one.
We remark that the existence of this solution is implicitly established in the proof of~\cite[Proposition 4.4]{AouG23}.
Let $A_\ell\in\RR^{\ell\times \ell}$ be the zero-one matrix such that
the entry in the $i$-th row and $j$-th column is equal to one iff $i+j\ge \ell+1$ and $\gcd(i,j)=1$.
For example, the matrices $A_4$ and $A_5$ are the following:
\[A_4=\begin{bmatrix} 0 & 0 & 0 & 1 \\ 0 & 0 & 1 & 0 \\ 0 & 1 & 0 & 1 \\ 1 & 0 & 1 & 0 \end{bmatrix}\quad\mbox{and}\quad
  A_5=\begin{bmatrix} 0 & 0 & 0 & 0 & 1 \\ 0 & 0 & 0 & 0 & 1 \\ 0 & 0 & 0 & 1 & 1 \\ 0 & 0 & 1 & 0 & 1 \\ 1 & 1 & 1 & 1 & 0 \end{bmatrix}.
\]
Since $\gcd(i,j) = \gcd(i, j + ti)$ for any $t \in \ZZ$,
it follows that $\phi(i)$ is the number of integers coprime with $i$ in any sequence of $i$ consecutive integers,
which yields that
\[
\sum_{j=1}^\ell (A_\ell)_{i,j} = \sum_{j=1}^\ell (A_\ell)_{j,i} = \sum_{\substack{j = \ell + 1 - i\\ \gcd(i,j) = 1}}^\ell{1} = \phi(i)
\]
for all $i\in\{1,\ldots,\ell\}$.
We conclude that $A_\ell$ is a feasible solution of \Al{}.

\begin{figure}
\begin{center}
\epsfbox{crtorus-4.mps}
\end{center}
\caption{Visualization of the relation of the matrices $A_5$ and $A_6$ in the proof of Lemma~\ref{lm:LP}.}
\label{fig:A56}
\end{figure}

We will show by induction on $h$ that $v_\ell^TA_\ell{}v_\ell=1$ for every $\ell\in\NN$.
If $\ell=1$, the statement clearly holds.
Suppose that $\ell>1$ and that the statement holds for $\ell-1$.
Observe that the following holds for $i,j\in\{1,\ldots,\ell\}$ (also see Figure~\ref{fig:A56}):
\begin{itemize}
\item if $i<\ell$, $j<\ell$ and either $i+j\ge \ell + 1$ or $i + j < \ell$, then $(A_{\ell-1})_{i,j} = (A_{\ell})_{i,j}$,
\item if $i+j=\ell$ or $i = j = \ell$, then $(A_\ell)_{i,j} = 0$,
\item the diagonal entry $(A_{\ell-1})_{i,\ell-i}$ is equal to $(A_\ell)_{i,\ell}$ and $(A_\ell)_{\ell,i}$ for every $i\in\{1,\ldots,\ell-1\}$.
\end{itemize}
The last property follows from the identity $\gcd(i, \ell) = \gcd(i, \ell - i)$. 
For brevity, set $a_i := (A_\ell)_{i,\ell} = (A_\ell)_{\ell,i}$ for $i\in\{1,\ldots,\ell-1\}$.
Using the three properties stated above,
we obtain that
\begin{align*}
v_\ell^TA_\ell{}v_\ell & = v_{\ell-1}^TA_{\ell-1}v_{\ell-1}-\sum_{i=1}^{\ell-1}\frac{a_i}{i(\ell-i)}+\sum_{i=1}^{\ell-1}\frac{a_i}{i\ell}+\sum_{i=1}^{\ell-1}\frac{a_i}{(\ell-i)\ell} \\
            & = 1+\sum_{i=1}^{\ell-1}a_i\left(\frac{1}{i\ell}+\frac{1}{(\ell-i)\ell}-\frac{1}{i(\ell-i)}\right) \\
	    & = 1+\sum_{i=1}^{\ell-1}a_i\frac{\ell-i+i-\ell}{i(\ell-i)\ell}=1.
\end{align*}
We conclude that the objective value of the feasible solution $A_\ell$ of \Al{} is at most one.

For $\ell\ge 4$, we show that $A_\ell$ can be perturbed to a feasible solution of \Al{} with objective value strictly smaller than one.
Observe that if $\ell\ge 4$,
then the entries $(A_\ell)_{\ell-2,\ell-1}$, $(A_\ell)_{\ell-1,\ell-2}$, $(A_\ell)_{\ell-1,\ell}$ and $(A_\ell)_{\ell,\ell-1}$ are equal to one.
Let $A'_\ell$ be the matrix obtained from $A_\ell$
by subtracting one from each of these four entries,
adding one to the entries $(A_\ell)_{\ell,\ell-2}, (A_\ell)_{\ell-2,\ell}$, and
adding two to the entry $(A_\ell)_{\ell-1,\ell-1}$.
For example, if $\ell=7$, we have the following:
\[A_7=\begin{bmatrix}
      0 & 0 & 0 & 0 & 0 & 0 & 1 \\
      0 & 0 & 0 & 0 & 0 & 0 & 1 \\
      0 & 0 & 0 & 0 & 1 & 0 & 1 \\
      0 & 0 & 0 & 0 & 1 & 0 & 1 \\
      0 & 0 & 1 & 1 & 0 & 1 & 1 \\
      0 & 0 & 0 & 0 & 1 & 0 & 1 \\
      1 & 1 & 1 & 1 & 1 & 1 & 0
      \end{bmatrix}
  \qquad\mbox{and}\qquad
  A'_7=\begin{bmatrix}
      0 & 0 & 0 & 0 & 0 & 0 & 1 \\
      0 & 0 & 0 & 0 & 0 & 0 & 1 \\
      0 & 0 & 0 & 0 & 1 & 0 & 1 \\
      0 & 0 & 0 & 0 & 1 & 0 & 1 \\
      0 & 0 & 1 & 1 & 0 & 0 & 2 \\
      0 & 0 & 0 & 0 & 0 & 2 & 0 \\
      1 & 1 & 1 & 1 & 2 & 0 & 0
      \end{bmatrix}
  .\]
Since the row and column sums of the matrices $A_\ell$ and $A'_\ell$ are the same,
the matrix $A'_\ell$ is a feasible solution of \Al{}.
The objective value of $A'_\ell$ is equal to
\begin{align*}
v_\ell^TA'_\ell{}v_\ell & = v_\ell^TA_\ell{}v_\ell-\frac{2}{(\ell-2)(\ell-1)}-\frac{2}{(\ell-1)\ell}+\frac{2}{(\ell-2)\ell}+\frac{2}{(\ell-1)^2}\\
&  = 1 - 2 \left( \frac{1}{\ell-2} - \frac{1}{\ell-1}  \right) \left( \frac{1}{\ell-1} - \frac{1}{\ell}  \right) <1.
\end{align*}
It follows that $A'_\ell$ is a feasible solution of \Al{} with objective value strictly smaller than one.
\end{proof}

\subsection{Upper bound on the height}

Since $\gamma_{\ell}<1$ for every $\ell\ge 4$ by Lemma~\ref{lm:LP},
Lemma~\ref{lm:size} asserts that $k$-nice sets with a fixed height of four or more has size less than $k$
when $k$ is sufficiently large.
We next combine Lemmas~\ref{lm:density}, \ref{lm:sqroot} and \ref{lm:size} to show that for all $k \geq 3225$,
there exist maximum size $k$-nice set with height at most three.

\begin{theorem}
\label{thm:k0}
For every $k\ge 3225$,
there exists a $k$-nice set of maximum size that has height at most three.
\end{theorem}

\begin{proof}
Let $Q$ be a $k$-nice set of maximum size.
Let $h_0$ be the height of $Q$; by Lemma~\ref{lm:sqroot},
we may assume that $h_0 \leq \sqrt{2k}$. 
If $Q$ has $k+2$ elements,
then there is nothing to prove since the set $\{(1,0),(0,1),(1,1),\ldots,(k,1)\}$
is a $k$-nice set with $k+2$ elements with height one.
Hence, we will assume that $Q$ has at least $k+3$ elements.

We now show that if $h_0 \ge 4$, then $|Q| < k+3$, a contradiction. To do so, we need to distinguish three cases based on the size of $h_0$.
\begin{itemize}
\item {\bf Case $h_0\ge 41020$.} 
      It is straightforward to verify that
      \[\frac{3264\pi}{10255}\cdot \frac{h_0^2}{2}+\frac{4946}{3675}\cdot h_0+1<\frac{h_0^2}{2}+3.\]
      In particular, since $k\ge h_0^2/2$ and $\frac{3264\pi}{10255}<1$, we obtain using Lemma~\ref{lm:density} that
      \[ |Q| \leq \frac{3264\pi}{10255}\cdot k+\frac{4946}{3675}\cdot h_0+1<k+3.\]
\item {\bf Case $h_0\in\{81,\ldots,41019\}$.}
      We have verified with computer assistance that
      \[\gamma_h\cdot\frac{h^2}{2}+\beta_h<\frac{h^2}{2}+3\]
      for every $h\in\{81,\ldots,41019\}$.
      Since $k\ge h_0^2/2$ and $\gamma_{h_0}<1$ (by Lemma~\ref{lm:LP}), we obtain using Lemma~\ref{lm:size} that
      \[|Q| \leq \gamma_{h_0} k+\beta_{h_0}\le \gamma_{h_0}\frac{h_0^2}{2}+\left(k-\frac{h_0^2}{2}\right)+\beta_{h_0}<k+3.\]
\item {\bf Case $h_0\in\{4,\ldots,80\}$.}
      We have verified with computer assistance that $3225\gamma_h+\beta_h<3225+3$ for every $h\in\{4,\ldots,80\}$.
      Since $\gamma_{h_0}<1$ (by Lemma~\ref{lm:LP}), we obtain using Lemma~\ref{lm:size} that
      \[
      |Q| \leq \gamma_{h_0} k+\beta_{h_0}\le\gamma_{h_0}3225+(k-3225)+\beta_{h_0}< k+3.
      \]
\end{itemize}
It follows that the height $h_0$ of the set $Q$ is at most three.
\end{proof}

We remark that the bound of $3225$ in Theorem~\ref{thm:k0} is the best possible in the setting of the proof:
consider $k=3224$, $h=80\le\sqrt{2k}$ and note that $\gamma_h k+\beta_h\approx 3227.039$, so Lemma~\ref{lm:size} does not rule out the existence of a $k$-nice set of size $k+3=3227$ and height $h=80$.

Our next aim is to improve the bound in Lemma~\ref{lm:sqroot} when $k \leq 3224$,
which will eventually lead to an improved version of Theorem~\ref{thm:k0}.
We do so with computer assistance employing Algorithm~\ref{alg:height},
which is analyzed in the next lemma.

\begin{algorithm}
{\bf Input:} positive integers $k\in\NN$ and $h\in\NN$ such that $2\le h\le k$\\
{\bf Output:} \emph{verified} or \emph{not verified}\\
\vskip 2ex
\nlset{$x_0$-loop}
\For{$x_0$ = $1$ \KwTo $\lfloor h/2\rfloor$}{
  \lIf{$\gcd(x_0,h)\not=1$}{\Continue}
  \lIf{$h\le\frac{k-x_0}{h}$}{\Return \emph{not verified}}
  \For{$y$ = $1$ \KwTo $h$}{
    \nlset{$x$-loop}
    \For{$x$ = $h$ \KwTo $\left\lfloor\frac{x_0 y+k}{h}\right\rfloor$}{
      \nlset{Continue 0}
      \lIf{$\gcd(x,y)\not=1$}{\Continue}
      $z$ = $\min\left\{y,x-y\right\}$\;
      \nlset{Continue 1}
      \lIf{$z+\frac{k}{x}<h$}{\Continue}
      $w$ = $1$\;
      \For{$y'$ = $1$ \KwTo $h$}{
        \For{$x'$ = $\left\lceil\frac{y'(x_0-h)-k}{h}\right\rceil$ \KwTo $\left\lfloor\frac{y'(x_0-h)+k}{h}\right\rfloor$}{
          \lIf{$\gcd(x',y')\not=1$}{\Continue}
	  \lIf{$|x'y-(x-y)y'|>k$}{\Continue}
          \lIf{$|x'|>w$}{$w$ = $|x'|$}
	  }
        }
      \nlset{Continue 2}
      \lIf{$w<h$}{\Continue}
      \Return \emph{not verified}\;
      }
    }
  }
\Return \emph{verified}\;
\vskip 2ex
\caption{Algorithm that verifies the following:
         every $k$-nice set with height $h$
	 is equivalent to a $k$-nice set with height less than $h$.}
\label{alg:height}
\end{algorithm}

\begin{lemma}
\label{lm:alg:height}
If Algorithm~\ref{alg:height} for input $k\in\NN$ and $h\in\NN$, $2\le h\le k$, returns \emph{verified},
then the following statement is true:
every $k$-nice set with height $h$
is equivalent to a $k$-nice set with height less than $h$.
\end{lemma}

\begin{proof}
Fix $k\in\NN$ and $h\in\NN$, $h\le k$, such that Algorithm~\ref{alg:height} returned \emph{verified}.
Let $Q$ be a $k$-nice set with height $h$.
By negating a subset of the elements of $Q$, we may assume that $Q$ is $y$-non-negative.
Since the height of $Q$ is $h$, there exists a point $(x_0,h)\in Q$.
By considering the set $A^mQ$ for some $m\in\ZZ$ instead of $Q$,
where $A$ is the matrix
\[A=\begin{bmatrix} 1 & 1 \\ 0 & 1 \end{bmatrix},\]
we may assume that $|x_0|\le h/2$; by applying  symmetry along the axis $x=0$ to $Q$ if needed, we can impose $0 < x_0 \le h/2$. 

First, observe that any point $(x,y)\in Q$ with $x<0$ and $y\ge 1$ satisfies
\[|x|\le \frac{k}{h}-\frac{x_0}{h}\cdot y\le \frac{k}{h}-\frac{x_0}{h}=\frac{k-x_0}{h}<h,\]
where the last inequality holds since the main \texttt{$x_0$-loop} in Algorithm~\ref{alg:height}
did not return \emph{not verified} when testing $h\le\frac{k-x_0}{h}$.
If the set $Q$ contained no point $(x,y)\in Q$ with $x\ge h$,
then the width of $Q$ would be less than $h$ and so the set $A'Q$,
where $A'$ is the matrix
\[A'=\begin{bmatrix} 0 & 1 \\ -1 & 0 \end{bmatrix},\]
would be a set equivalent to $Q$ with height at most $h-1$.
Hence, we may assume that the set $Q$ contains a point $(x,y)\in Q$ with $x\ge h$, and
choose such a~point so that $x$ is as large as possible.
Note that $xh-x_0y\le k$ (as the set $Q$ is $k$-nice) and so $x\le\frac{x_0y+k}{h}$.
Also note that
the point $(x,y)$ is also a point with the largest first coordinate in the absolute value (as
the first coordinate of all points with a negative first coordinate is greater than $-h$).

Since Algorithm~\ref{alg:height} did not return \emph{not verified},
the \texttt{$x$-loop} for $(x,y)$ executed either \texttt{Continue 0}, \texttt{Continue 1} or \texttt{Continue 2}.
As $(x,y)\in Q$, we must have $\gcd(x,y)=1$ and
so either \texttt{Continue 1} or \texttt{Continue 2} was executed in the \texttt{$x$-loop}.

We first analyze the case where \texttt{Continue 1} was executed in the \texttt{$x$-loop}.
We distinguish two cases: $y\le x-y$ and $y>x-y$.
\begin{itemize}
\item If $y\le x-y$, then it holds that $y+\frac{k}{x}<h$ (because \texttt{Continue 1} was executed).
      Since the set $Q$ is $k$-nice,
      the point $(x_0,h)$, which is contained in $Q$, satisfies that $xh-x_0y\le k$ and
      so it holds that
      \[h\le \frac{x_0}{x}y+\frac{k}{x}\le y+\frac{k}{x}<h,\]
      which is impossible.
\item If $y>x-y$, we consider the set $Q'$ obtained from the set $A'Q$
      by negating all points with negative second coordinate.
      Note that the set $Q'$ is $y$-non-negative and $k$-nice,
      the set $Q'$ contains the point $(-y,x)$, and
      the second coordinate of any point in $Q'$ is at most $x$ (recall that
      the point $(x,y)$ is a point with the largest first coordinate in the absolute value in $Q$).

      We next consider the set $AQ'$, which contains the point $(x-y,x)$;
      note that $x-y$ is non-negative as $x\ge h$ and $y\le h$.
      As \texttt{Continue 1} was executed and $x-y<y$, it holds that $(x-y)+\frac{k}{x}<h$.
      Since the set $AQ'$ is $k$-nice,
      every point $(x',y')$ of $Q'$ satisfies that
      \[x'+y'\le \frac{y'}{x}(x-y)+\frac{k}{x}\le x-y+\frac{k}{x}<h.\]
      On the other hand, every point $(x',y')$ of $Q'$ also satisfies that
      $x'+y'>x'\ge -h$ (we use that the height of $Q$ is $h$) unless $y'=0$.
      Note that the only point $(x',y')$ contained in $Q'$ with $y'=0$ can be the point $(-1,0)$.
      We conclude that
      every point $(x',y')$ contained in the set $Q'$ satisfies that $-h<x'+y'<h$,
      i.e.~the width of the set $AQ'$ is at most $h-1$.
      Consequently, the height of the set $A'AQ'$, which is equivalent to the original set $Q$, is less than $h$.
\end{itemize}
The analysis of the case where \texttt{Continue 1} was executed is now finished.

We next analyze the case where \texttt{Continue 2} was executed in the \texttt{$x$-loop}. Since
\[A^{-1}=\begin{bmatrix} 1 & -1 \\ 0 & 1 \end{bmatrix},\]
observe that the set $A^{-1}Q$ contains the points $(x_0-h,h)$ and $(x-y,y)$.
It follows that every point $(x',y')$ contained in the set $A^{-1}Q$ satisfies that
\[\frac{y'(x_0-h)-k}{h}\le x'\le\frac{y'(x_0-h)+k}{h}\quad\mbox{and}\quad
  |x'y-(x-y)y'|\le k;\]
in addition, the point $(x',y')$ also satisfies that $\gcd(x',y')=1$.
Since \texttt{Continue 2} was executed in the \texttt{$x$-loop},
it follows that $|x'|<h$ for all points $(x',y')\in A^{-1}Q$,
i.e., the width of $A^{-1}Q$ is at most $h-1$ and
so the height of the set $A'A^{-1}Q$ is at most $h-1$.
\end{proof}

We are now ready to improve Lemma~\ref{lm:sqroot} for $k\in\{2,\ldots,3224\}$.
We remark that the multiplicative constant $\sqrt{4/3}$ obtained in the following lemma
cannot be improved in general as there exists a $3$-nice set with $6$ elements and height $2$,
but every $3$-nice set with height one has at most $5$ elements.

\begin{lemma}
\label{lm:sqroot-alg}
Let $k\in\{2,\ldots,3224\}$.
For every $k$-nice set $Q$,
there exists a $k$-nice set equivalent to $Q$ that has height at most $\sqrt{4k/3}$.
\end{lemma}

\begin{proof}
We executed Algorithm~\ref{alg:height} for all $k\in\{2,\ldots,3224\}$ and
all $h\in\NN$ such that $\sqrt{4k/3}<h\le\sqrt{2k}$, and
Algorithm~\ref{alg:height} always returned \emph{verified} (the implementation of Algorithm~\ref{alg:height}
is available as an ancillary file named \verb|Algorithm1.c| with the arXiv version of this manuscript).

Consider a $k$-nice set $Q$ and
let $Q_0$ be a set equivalent to $Q$ with the smallest possible height $h$.
By Lemma~\ref{lm:sqroot}, $h\le\sqrt{2k}$.
If $h>\sqrt{4k/3}$, we obtain a contradiction by Lemma~\ref{lm:alg:height} applied with $k$ and $h$.
It follows that $h\le\sqrt{4k/3}$.
\end{proof}

We are now ready to prove the main theorem of this section.

\begin{theorem}
\label{thm:k0new}
For every $k\ge 1892$,
there exists a $k$-nice set of maximum size that has height at most three.
\end{theorem}

\begin{proof}
Let $Q$ be a $k$-nice set of maximum size.
We can assume that $|Q|\ge k+3$;
if $|Q|=k+2$, then the set $\{(1,0),(0,1),(1,1),\ldots,(k,1)\}$ is an example of a $k$-nice set with height one. If $k\ge 3225$, the statement of the theorem follows from Theorem~\ref{thm:k0}, so we henceforth assume that $k\in\{1892,\ldots,3224\}$. By Lemma~\ref{lm:sqroot-alg}, we may assume that the height $h_0$ of $Q$ is at most $\sqrt{4k/3}$, so that $h_0\le 66$.

As in the proof of Theorem~\ref{thm:k0},
we now distinguish two cases based on the size of $h_0$ in order to show that if $h_0 \ge 4$, then $|Q| < k+3$, a contradiction.
\begin{itemize}
\item {\bf Case $h_0\in\{51,\ldots,66\}$.}
      We have verified with computer assistance that
      \[\gamma_h\cdot\frac{3h^2}{4}+\beta_h<\frac{3h^2}{4}+3\]
      for every $h\in\{51,\ldots,66\}$.
      Since $k\ge 3h_0^2/4$ and $\gamma_{h_0}<1$ (by Lemma~\ref{lm:LP}),
      we obtain using Lemma~\ref{lm:size} that
      the size of $Q$ is at most $\gamma_{h_0} k+\beta_{h_0}<k+3$.
\item {\bf Case $h_0\in\{4,\ldots,50\}$.}
      We have verified with computer assistance that $1892\gamma_h+\beta_h<1892+3$ for every $h\in\{4,\ldots,50\}$.
      Since $\gamma_{h_0}<1$ (by Lemma~\ref{lm:LP}), we obtain using Lemma~\ref{lm:size} that
      the size of $Q$ is at most $\gamma_{h_0} k+\beta_{h_0}<k+3$.
\end{itemize}
It follows that $h_0\le 3$, which completes the proof of the theorem.
\end{proof}

Again, the bound of $1892$ in Theorem~\ref{thm:k0new} is best possible in the setting of the proof, since for $k=1891$ and $h=50\le\sqrt{4k/3}$, we obtain that $\gamma_h k+\beta_h\approx 1894.036$.

\section{Nice sets with height at most three}
\label{sec:three}

This section is devoted to proving the following lemma,
which determines the maximum size of a $k$-nice set with height at most three.
We remark that the constructions witnessing that the bounds in the statement of the lemma
are tight already appeared in~\cite{GasLRRX20};
we include them in the proof of the lemma for completeness of the argument.

\begin{lemma}
\label{lm:height123}
For every $k\ge 3$,
the maximum size of a $k$-nice set of height at most~$3$ is
\begin{itemize}
\item $k+4$ if $k \bmod 6 = 2$,
\item $k+3$ if $k \bmod 6\in\{1,3,5\}$, and
\item $k+2$ otherwise.
\end{itemize}
\end{lemma}

\begin{proof}
Fix $k\ge 3$ and let $N_{k,h}$ be the maximum size of a $k$-nice $y$-non-negative set with height $h\in\{1,2,3\}$.
The statement of the lemma follows from the following four claims, which we prove next (note that
every $k$-nice set is equivalent to a $y$-non-negative set with the same height).

\begin{claim} \label{smallh1}
$N_{k,1} = k+2$.
\end{claim}
\begin{claim} \label{smallh2}
$N_{k,2} \leq k+3$, with equality if and only if $k \equiv 1 \bmod 2$.
\end{claim}
\begin{claim} \label{smallh3}
 $N_{k,3} \leq k+4$, with equality if and only if $k \equiv 2 \bmod 6$.
\end{claim}
\begin{claim} \label{smallh4}
$N_{k,3} \leq k+2$ if $k \equiv 0 \bmod 6$ or $k \equiv 4 \bmod 6$.
\end{claim}

In what follows, when $Q$ is a $k$-nice $y$-non-negative set with height $h$ and $i\in\{0,\ldots,h\}$,
we let $Q_i = \{x : (x,i) \in Q\}$ and,
if $Q_i$ is non-empty, we let $s_i$ and $t_i$ denote the smallest and largest elements of $Q_i$, respectively.  
Note that $|Q_0|\le 1$ as $Q_0$ may only contain either $1$ or $-1$.

\begin{proof}[Proof of Claim~\ref{smallh1}]
Consider a $k$-nice $y$-non-negative set $Q$ with height one.
Since the set $Q$ is $k$-nice, $(s_1,1)\in Q$ and $(t_1,1)\in Q$,
we obtain that $t_1 - s_1 \leq k$ and so $|Q_1| \leq k+1$.
It follows that $|Q| = |Q_0| + |Q_1| \leq k+2$.
This implies that $N_{k,1}\le k+2$.
The bound is attained by the set $Q = \{(1,0)\} \cup \{(x,1) : 0 \leq x \leq k\}$.
\end{proof}

\begin{proof}[Proof of Claim~\ref{smallh2}]
Consider a $k$-nice $y$-non-negative set $Q$ with height two.
If $Q_1$ is empty, we obtain $2t_2-2s_2\le k$ using that $Q$ is $k$-nice, $(s_2,2)\in Q$ and $(t_2,2)\in Q$.
So, the size of $Q_2$ is at most $k/2+1$ and so the size of $Q$ is at most $|Q_0|+|Q_2|\le 1+k/2+1$, which yields that $|Q|\le k+2$.
Hence, we can assume that $Q_1$ is non-empty and so $Q$ contains the points $(s_1, 1)$ and $(t_1, 1)$.

Since the set $Q$ is $k$-nice,
we get that $t_2 - 2 s_1 \leq k$ and $2t_1 - s_2 \leq k$,
which yields that $s_1 \geq \left \lceil \frac{t_2 - k}{2}\right \rceil$ and $t_1 \leq \left \lfloor \frac{s_2 + k}{2}\right \rfloor$. 
It follows that
\[|Q_1| \leq t_1 - s_1 + 1 \leq \left \lfloor \frac{s_2 + k}{2}\right \rfloor - \left \lceil \frac{t_2 - k}{2}\right \rceil + 1.\]
Since the set $Q_2$ consists of odd numbers only,
it follows that $Q_2$ has at most $\frac{t_2 - s_2}{2} + 1$ elements. We obtain using $|Q_0| \leq 1$ that 
\begin{align*}
|Q|=|Q_0|+|Q_1|+|Q_2| 
& \leq 1+\left \lfloor \frac{s_2 + k}{2}\right \rfloor - \left \lceil \frac{t_2 - k}{2}\right \rceil + 1 \frac{t_2 - s_2}{2} + 1 \\
& \leq \frac{s_2 + k}{2} -  \frac{t_2 - k}{2} + \frac{t_2 - s_2}{2} + 3 = k+3.
\end{align*}
It follows that the size of $Q$ is at most $k + 3$, with equality only if the above estimate is tight, which can only occur
when $s_2+k$ is divisible by two, i.e., when $k\equiv 1 \bmod 2$ (recall that $s_2$ is odd).
This establishes that $N_{k,2}\le k+3$ if $k\equiv 1 \bmod 2$, and $N_{k,2}\le k+2$ otherwise.
Moreover, if $k \equiv 1 \bmod 2$, then the bound in Claim~\ref{smallh2} is tight
as the set $Q = \{(1,0), (k,2)\} \cup \{(x,1): 0 \leq x \leq k\}$ is $k$-nice.
\end{proof}

\begin{proof}[Proof of Claim~\ref{smallh3}]
Consider a $k$-nice $y$-non-negative inclusion-wise maximal set $Q$ with height three.
Observe that since $Q_3$ consists only of numbers that are not divisible by 3,
we can compute the size of $Q_3$ as
\begin{equation}
|Q_3| = \frac{2}{3}\left(t_3 - s_3\right) + 1 + \rho,
\label{eq:Q3}
\end{equation}
where the constant $\rho$ is (recall that neither $s_3$ nor $t_3$ is divisible by three)
\begin{equation*}
\rho = \begin{cases}
0 &\text{ if } t_3 \equiv s_3 \bmod 3,\\
1/3 &\text{ if } s_3 \equiv 1\bmod 3\text{ and } t_3 \equiv 2 \bmod 3,\text{ and}\\
-1/3 &\text{ if } s_3 \equiv 2\bmod 3\text{ and } t_3 \equiv 1 \bmod 3.
\end{cases}
\end{equation*} 

We first argue that $|Q| \leq k+2$ unless $Q_1$ and $Q_2$ are both non-empty.
Since the set $Q$ is $k$-nice,
it holds that $3t_3-3s_3 \leq k$ and $2t_2 - 2s_2 \leq k$ (assuming $Q_2$ is non-empty and so $s_2$ and $t_2$ are defined),
which implies that
$|Q_3| \leq \frac 23 (t_3 - s_3) + 1 + \rho \leq \frac 29k + \frac 43$ and
$|Q_2| \leq \frac 12(t_2 - s_2) + 1 \leq \frac k4 + 1$ (note that this estimate also holds if $Q_2$ is empty).
Therefore,
if $Q_1$ is empty,
we obtain (using $k\ge 3$) that
the set $Q$ has at most $|Q_0| + |Q_2| + |Q_3| \le \frac {17}{36}k + \frac{10}{3} \le k + 2$ elements.
We now consider the case that $Q_1$ is non-empty and $Q_2$ is empty.
If $|Q_1|+|Q_3|\le 4$, then the size of $Q$ is at most $|Q_0|+|Q_1|+|Q_3|\le 5\le k+2$.
If $|Q_1|+|Q_3|\ge 5$,
then $t_1-s_1\ge |Q_1|-1$ and $t_3-s_3\ge |Q_3|-1$ and,
since the convex hull of $Q$ contains a~line segment with endpoints of the form $((s_1+s_3)/2, 2)$ and $((t_1+t_3)/2, 2)$ (note that
the length of this segment is at least $1.5$ and its endpoints are half-integral),
the convex hull contains points $(m,2)$ and $(m+1,2)$ for some $m\in\ZZ$.
In particular, since $Q$ is inclusion-wise maximal,
either $(m,2)$ or $(m+1,2)$ belongs to $Q$, which contradicts $Q_2$ being empty.

In the rest of the proof of the claim, we assume that both $Q_1$ and $Q_2$ are non-empty.
Since the set $Q$ is $k$-nice, the following inequalities hold:
\[
t_3 - 3s_1 \leq k, \quad
3t_1 - s_3 \leq k,\quad
2t_3 - 3s_2 \leq k\quad\text{and}\quad
3t_2 - 2 s_3 \leq k.
\]
We now define the following shorthand notation:
\begin{align*}
a_1 & = \left \lceil \frac{t_3 - k}{3} \right \rceil -  \frac{t_3 - k}{3},&
b_1 & = \frac{s_3 + k}{3} - \left \lfloor \frac{s_3 + k}{3} \right \rfloor,\\ 
a_2 & = \left \lceil \frac{2t_3 - k}{3} \right \rceil - \frac{2t_3 - k}{3}\qquad\text{and}&
b_2 & = \frac{2s_3 + k}{3} - \left \lfloor \frac{2s_3 + k}{3} \right \rfloor.
\end{align*}
Note that all the four quantities $a_1$, $b_1$, $a_2$ and $b_2$ are non-negative.
Moreover, since $t_3$ is not divisible by three, $a_1$ and $a_2$ cannot both be equal to zero, and
since $s_3$ is not divisible by three, $b_1$ and $b_2$ cannot both be equal to zero.
Therefore, it holds that $\min\{a_1,a_2\}\ge 1/3$ and $\min\{b_1,b_2\}\ge 1/3$,
which implies that
\begin{equation}
a_1 + b_1 + \frac{a_2}{2} + \frac{b_2}{2}\ge \frac{a_1+a_2}{2}+\frac{b_1+b_2}{2}\ge \frac{1}{3}.\label{eq:a12b12}
\end{equation}
We now obtain the following estimates on $s_1$, $t_1$, $s_2$ and $t_2$:
\begin{align}
s_1 & \geq \frac{t_3 - k}{3} + a_1, &
t_1 & \leq \frac{s_3 + k}{3} - b_1, \nonumber\\
s_2 & \geq \frac{2t_3 - k}{3} + a_2 \qquad\text{and}&
t_2 & \leq \frac{2s_3 + k}{3} - b_2. \label{eq:s2t2}
\end{align}
Using these four estimates, we obtain that
\begin{align*}
|Q_1| &= t_1 - s_1 + 1 \leq \frac{s_3 + k}{3} - \frac{t_3 - k}{3} + 1 - a_1 - b_1,\\
|Q_2| &= \frac{t_2 - s_2}{2} + 1 \leq \frac{1}{2}\left( \frac{2s_3 + k}{3} -  \frac{2t_3 - k}{3} - a_2 - b_2 \right) + 1,
\end{align*}
These two estimates on the sizes of $Q_1$ and $Q_2$ and the estimate \eqref{eq:Q3} now yields that
\begin{equation}
|Q| = |Q_0| + |Q_1| + |Q_2| + |Q_3| 
\leq k + 4 - \left( a_1 + b_1 + \frac{a_2}{2} + \frac{b_2}{2} \right)+\rho,
\label{eq:finalQ}
\end{equation}
which implies that the size of $Q$ is at most $k+4$ (we use \eqref{eq:a12b12} and $\rho\le 1/3$).

Suppose that $|Q| = k + 4$.
Recall that $a_1$ or $a_2$ is at least $1/3$ and $b_1$ or $b_2$ is at least $1/3$.
Therefore, the right side of \eqref{eq:finalQ} is smaller than $k+4$
unless $a_1=b_1=0$, $a_2=b_2=1/3$, $\rho=1/3$ and all four inequalities in \eqref{eq:s2t2} are tight.
It follows from the definition of $\rho$ that $s_3 \equiv 1 \bmod 3$ and $t_3 \equiv 2 \bmod 3$.
We next derive that $k \equiv t_3 \equiv 2\bmod 3$ (as $a_1=0$ and so $k=t_3+3s_1$) and
$k \equiv s_2+1 \equiv 0 \bmod 2$ (as $a_2=1/3$ and so $k=2t_3-3s_2+1$; recall that $s_2$ is odd).
We conclude that if $|Q| = k + 4$, then $k \equiv 2 \bmod 6$. 

On the other hand,
if $k \equiv 2 \bmod 6$,
we can construct a $k$-nice set $Q$ with height three and $k+4$ elements as follows.
Let $s$ be the smallest integer larger than $\frac{2}{3}k$ satisfying $s \equiv 1 \bmod 3$, and
consider the $k$-nice set
\begin{align*}
Q = & \{(1,0)\} \cup \left\{(0,1), \ldots, \left(\frac{k+s}{3},1\right) \right\} \cup\\
    & \left\{\left(\frac{k+1}{3},2\right), \ldots, \left(\frac{k+2s-1}{3}, 2\right) \right\} \cup \{(s,3), \ldots, (k,3)\},
\end{align*}
which has $1+\frac{k+s}{3}+1+\frac{k+2s-1-(k+1)}{6}+1+\frac{2(k-s+2)}{3}=k+4$ elements.
\end{proof}

\begin{proof}[Proof of Claim~\ref{smallh4}]
Consider a $k$-nice $y$-non-negative inclusion-wise maximal set $Q$ with height three.
If $Q_1$ or $Q_2$ is empty, then $Q$ has at most $k+2$ elements as argued in the proof of Claim~\ref{smallh3}.
We will need a refined version of \eqref{eq:finalQ}.
Define $c_2$ as
\[
c_2 = \begin{cases}
1 &\text{ if } \left \lfloor \frac{2s_3 + k}{3} \right \rfloor \text{is even, and}\\
0 &\text{ otherwise.}
\end{cases}
\]
Since $t_2$ is odd, we can strengthen the estimate \eqref{eq:s2t2} on $t_2$ to
\[t_2 \leq \frac{2s_3 + k}{3} - b_2 - c_2,\]
which leads to the following stronger version of \eqref{eq:finalQ}:
\begin{equation}
|Q| \leq k + 4 - \left( a_1 + b_1 + \frac{a_2}{2} + \frac{b_2}{2} + \frac{c_2}{2} \right)+\rho.
\label{eq:finalQrefined}
\end{equation}
The values of the quantities $a_1$, $b_1$, $a_2$, $b_2$, $c_2$ and $\rho$
for $k\equiv 0 \bmod 6$ and $k\equiv 4\bmod 6$ and
all possible values $s_3\not\equiv 0\bmod 3$ and $t_3\not\equiv 0\bmod 3$
can be found in Table~\ref{tab:values},
where $S=a_1 + b_1 + \frac{a_2}{2} + \frac{b_2}{2} + \frac{c_2}{2}$.
Since the value of $S-\rho=a_1 + b_1 + \frac{a_2}{2} + \frac{b_2}{2} + \frac{c_2}{2} - \rho$ in each of the cases
is larger than one, it follows that the size of $Q$ is less than $k+3$, i.e., it is most $k+2$.
\begin{table}
\begin{center}
\begin{tabular}{|c|c|c | c | c | c | c | c | c | c |}
\hline
$k \bmod 6$ & $s_3 \bmod 6$ & $t_3 \bmod 3$ & $a_1$ & $b_1$ & $a_2$ & $b_2$ & $c_2$ & $S$ & $\rho$ \\
\hline
 0 & 1 or 4 & 1 & $2/3$ & $1/3$ & $1/3$ & $2/3$ & 1 & $2$ & 0 \\
 0 & 1 or 4 & 2 & $1/3$ & $1/3$ & $2/3$ & $2/3$ & 1 & $11/6$ & $+1/3$ \\
 0 & 2 or 5 & 1 & $2/3$ & $2/3$ & $1/3$ & $1/3$ & 0 & $5/3$ & $-1/3$ \\
 0 & 2 or 5 & 2 & $1/3$ & $2/3$ & $2/3$ & $1/3$ & 0 & $3/2$ & 0 \\
\hline 
 4 & 1 or 4 & 1 & 0 & $2/3$ & $2/3$ & 0 & 1 & $3/2$ & 0 \\
 4 & 1 or 4 & 2 & $2/3$ & $2/3$ & 0 & 0 & 1 & $11/6$ & $+1/3$ \\
 4 & 2 or 5 & 1 & 0 & 0 & $2/3$ & $2/3$ & 1 & $7/6$ & $-1/3$ \\
 4 & 2 or 5 & 2 & $2/3$ & 0 & 0 & $2/3$ & 1 & $3/2$ & 0 \\
\hline
\end{tabular}
\end{center}
\caption{The values of the quantities $a_1$, $b_1$, $a_2$, $b_2$, $c_2$,
         $S=a_1 + b_1 + \frac{a_2}{2} + \frac{b_2}{2} + \frac{c_2}{2}$ and $\rho$
         in the cases considered in the proof of Claim~\ref{smallh4}.}
\label{tab:values}
\end{table}
\end{proof}
We now combine the claims to complete the proof of Lemma~\ref{lm:height123}.
Since a $k$-nice $y$-non-negative set with height 0 can only have at most 1 element,
we get $N_k = \max\{ N_{k,h} : h \in \{1, 2, 3\} \}$ and
the desired result then follows from Claims \ref{smallh1}, \ref{smallh2}, \ref{smallh3}, and \ref{smallh4}.
\end{proof}

\section{Algorithm}
\label{sec:algorithm}

In this section, we present our algorithm for computing the maximum size of a $k$-nice set,
which we use for computing the maximum size of a $k$-nice set for $k\in\{3,\ldots,1891\}$.
Before we do so, we need to establish the following lemma,
which concerns the structure of a $k$-nice set that we may assume in the algorithm.

\begin{lemma}
\label{lm:xynon}
Let $Q\subseteq\ZZ^2$ be an inclusion-wise maximal $k$-nice set with height $h$, $1\le h\le k$.
There exists an equivalent $k$-nice set $Q'$ with height $h$ such that
\begin{itemize}
\item $(1,0)\in Q'$, $(0,1)\in Q'$ and $(1,1)\in Q'$, and
\item $Q'\subseteq \{0,\ldots,k\}\times\{0,\ldots,h\}$.
\end{itemize}
\end{lemma}

\begin{proof}
Consider an inclusion-wise maximal $k$-nice set $Q$ with height $h$, $h\le k$.
We may assume that $Q$ is $y$-non-negative by negating a subset of its elements.
Since the height of $Q$ is at most $k$ and $Q$ is an inclusion-wise maximal $k$-nice set,
the set $Q$ contains one of the points $(1,0)$ or $(-1,0)$;
by negating the point if needed, we may assume that $(1,0)\in Q$.

Since the set $Q$ is $y$-non-negative, there exists $m\in\ZZ$ such that
the set $A^mQ$ is $x$-non-negative,
where $A$ is the matrix
\[A=\begin{bmatrix} 1 & 1 \\ 0 & 1 \end{bmatrix};\]
choose the smallest (possibly negative) $m\in\ZZ$ with this property.
Note that $m$ is well-defined since the set $Q$ contains at least one element with positive second coordinate (otherwise,
$Q$ would not be an inclusion-wise maximal $k$-nice set).

Set $Q'=A^mQ$.
Since $Q$ is an inclusion-wise maximal $k$-nice set,
the set $Q'$ is also an inclusion-wise maximal $k$-nice set.
The choice of $m$ implies that there exists $(x_0,y_0)\in Q'$ such that $x_0-y_0<0$ (otherwise,
the set $A^{m-1}$ would also be $x$-non-negative).
We claim that the first coordinate of any point in $Q'$ is at most $k$.
Suppose that $Q'$ contains a point $(x,y)$ such that $x>k$;
note that $y$ is at most $k$ as the height of $Q$ is at most $k$.
It follows (note that $y\le k<x$ and $x_0<y_0$) that
\[|x_0y-y_0x|=xy_0-x_0y\ge x(x_0+1)-x_0k=x+x_0(x-k)\ge x>k,\]
which is impossible since the set $Q'$ is $k$-nice.
We conclude that $Q'\subseteq\{0,\ldots,k\}\times\{0,\ldots,h\}$.
Finally, since the set $Q'$ is an inclusion-wise maximal $k$-nice set,
it contains the point $(0,1)$ (because $|x|\le k$ for every $(x,y)\in Q'$) and
the point $(1,1)$ (because $|y-x|\le k$ for every $(x,y)\in Q'$).
\end{proof}

We are now ready to present and analyze the recursive algorithm for computing
the maximum size of a $k$-nice set with given height.

\begin{algorithm}
\SetKwProg{Proc}{procedure}{}{end}
\renewcommand{\ProgSty}[1]{\texttt{#1}}
\Proc{compute($k$,$h$,$N$)}{
  $L[1]$ = 0; $U[1]$ = $k$\;
  \For{$i$ = $2$ \KwTo $h$}{
    $L[i]$ = 1; $U[i]$ = $k$;
    }
  \Return \texttt{backtrack}($k$,$h$,$N$,$h$,$0$,$L[1\bdots h]$,$U[1\bdots h]$)\;  
  }
\vskip 2ex
\caption{The procedure \texttt{compute}, which uses the recursive procedure \texttt{backtrack}
         given as Algorithm~\ref{alg:backtrack}.}
\label{alg:compute}
\end{algorithm}

\begin{algorithm}
\SetKwProg{Proc}{procedure}{}{end}
\renewcommand{\ProgSty}[1]{\texttt{#1}}
\vskip 2ex
\Proc{backtrack($k$,$h$,$N$,$\ell$,$N_{>}$,$L[1\bdots\ell]$,$U[1\bdots\ell]$)}{
  $M[i][a][b]$ = $|\{z, a\le z\le b \mbox{ and } \gcd(z,i)=1\}|$\;
  $M_k[i][a][b]$ = $\max\limits_{a\le a'\le b'\le b, |b'-a'|\le k/i} M[i][a'][b']$\;
  \nlset{Update N}
  \lIf{$\ell=0$}{\Return $N_{>}+1$}
  \If{$\ell<h$}{
    $N'$=$N_{>}+1$\;
    \For{$i$ = $1$ \KwTo $\ell-1$}{
      \lIf{$L[i]\le U[i]$}{$N'$ = $N'+M_k[i][L[i]][U[i]]$}
      }
    \lIf{$N'>N$}{\nlset{Call 1}$N$=\texttt{backtrack}($k$,$h$,$N$,$\ell-1$,$N_{>}$,$L[1\bdots\ell-1]$,$U[1\bdots\ell-1]$)}
    }
  \lIf{$L[\ell]>U[\ell]$}{\Return $N$}
  \For{$a$ = $L[\ell]$ \KwTo $U[\ell]$}{
    \For{$b$ = $a$ \KwTo $U[\ell]$}{
      \lIf{$\gcd(a,\ell)\not=1$ \textnormal{\textbf{or}} $\gcd(b,\ell)\not=1$}{\Continue}
      \lIf{$(b-a)\ell>k$}{\Continue}
      $N'$=$N_{>}+M[\ell][a][b]+1$\;
      \For{$i$ = $1$ \KwTo $\ell-1$}{
        $L'[i]$ = $\max\left\{L[i],\left\lceil\frac{ai-k}{\ell}\right\rceil,\left\lceil\frac{bi-k}{\ell}\right\rceil\right\}$\;
	$U'[i]$ = $\min\left\{U[i],\left\lfloor\frac{ai+k}{\ell}\right\rfloor,\left\lfloor\frac{bi+k}{\ell}\right\rfloor\right\}$\;
        \lIf{$L'[i]\le U'[i]$}{$N'$ = $N'+M_k[i][L'[i]][U'[i]]$}
        }
      \lIf{$N'\le N$}{\Continue}
      $N_{\ge}$ = $N_{>}+M[\ell][a][b]$\;
      \nlset{Call 2}
      $N$ = \texttt{backtrack}($k$,$h$,$N$,$\ell-1$,$N_{\ge}$,$L'[1\bdots\ell-1]$,$U'[1\bdots\ell-1]$)\;
      }
    }
  \Return $N$\;
  }
\vskip 2ex
\caption{The recursive procedure \texttt{backtrack}.}
\label{alg:backtrack}
\end{algorithm}

\begin{lemma}
\label{lm:main_correct}
Let $k\in\NN$ and $h\in\NN$ such that $2\le h\le k$.
For every $N\in\NN$,
the procedure \texttt{compute(k,h,N)} given as Algorithm~\ref{alg:compute} returns 
\begin{itemize}
\item the maximum size of a $k$-nice set with height $h$
      if there exists a $k$-nice set with height $h$ with more than $N$ elements, and
\item the value $N$, otherwise.
\end{itemize}
\end{lemma}

\begin{proof}
We analyze the code of the procedure \texttt{compute(k,h,N)} given as Algorithm~\ref{alg:compute} and
the recursive procedure \texttt{backtrack} given as Algorithm~\ref{alg:backtrack}.
Note that the recursive calls within the procedure \texttt{backtrack} are made only
on the lines marked as \texttt{Call 1} and \texttt{Call 2}.
We observe that the parameters $k$ and $h$ during the recursive calls of the procedure \texttt{backtrack} never change, and
the parameter $\ell$ always decreases by one.
In addition, the main body of the procedure is executed only if $\ell>0$.
We fix $k\in\NN$ and $h\in\NN$, $2\le h\le k$, for the rest of the proof.

We now prove by induction on the value of $\ell$ that
the return value of the procedure \texttt{backtrack} is always at least the parameter $N$ that it was called with;
moreover, the value of $N$ never decreases during the execution of the procedure \texttt{backtrack}.

If $\ell=0$, the procedure just returns $N_{>}+1$ on the line marked \texttt{Update N}.
The procedure \texttt{backtrack} was called from the instance for $\ell=1$
either at the line marked \texttt{Call 1} or at the line marked \texttt{Call 2}.
In the former case, it held that $N'=N_{>}+1$ is larger than $N$ (note the \texttt{if} condition just before \texttt{Call 1}) and
so the return value of the instance of the procedure for $\ell=0$ is larger than $N$.
In the latter case, it held that $N'=N_{>}+M[\ell][a][b]+1$ is larger than $N$ (note the \texttt{if} condition before \texttt{Call 2}) and
since the instance of the procedure for $\ell=0$ was called with $N_{\ge}=N_{>}+M[\ell][a][b]$,
its return value is larger than $N$.

If $\ell>0$, the value of $N$ is only affected by recursive calls of the procedure \texttt{backtrack},
however, they never return a smaller value of $N$ than the one that they were called with by induction.
We conclude that the return value of the procedure \texttt{backtrack} is always at least the parameter $N$.
Moreover, if the return value is larger,
then the return command on the line marked \texttt{Update N} was executed during the recursion and
the return value is actually equal to the largest value ever returned on the line marked \texttt{Update N}.

Consider an instance of the procedure \texttt{backtrack} for $\ell_0$;
note that the instance is at depth $h-\ell_0$ of the recursion.
Let $I_0$ be the values of $\ell$ in the instances that made a recursive call on the line marked \texttt{Call 1} and
let $a_i$ and $b_i$ be the values of the variables $a$ and $b$ when
a recursive call was made on the line marked \texttt{Call 2} for the value of $\ell$ equal to $i\in\{\ell_0+1,\ldots,h\}\setminus I_0$.
In particular, if $\ell_0=h$, then $I_0=\emptyset$ and no $a_i$'s and $b_i$'s are defined.
Note that $h\not\in I_0$.
Also, note that $(b_i-a_i)i\le k$ for every $i\in\{\ell_0+1,\ldots,h\}\setminus I_0$ and
it holds that
\[N_{>}=\sum_{j\in\{\ell_0+1,\ldots,h\}\setminus I_0}M[j][a_j][b_j].\]
Observe that the values of the array $L$ and $U$ satisfy the following for every $i\in\{1,\ldots,\ell_0\}$:
\begin{align*}
L[i]&=\max\left\{1,
          \max_{j\in\{\ell_0+1,\ldots,h\}\setminus I_0}\left\lceil\frac{ia_j-k}{j}\right\rceil,
          \max_{j\in\{\ell_0+1,\ldots,h\}\setminus I_0}\left\lceil\frac{ib_j-k}{j}\right\rceil\right\}\\
U[i]&=\min\left\{k,
          \min_{j\in\{\ell_0+1,\ldots,h\}\setminus I_0}\left\lfloor\frac{ia_j+k}{j}\right\rfloor,
          \min_{j\in\{\ell_0+1,\ldots,h\}\setminus I_0}\left\lfloor\frac{ib_j+k}{j}\right\rfloor\right\}\\
\end{align*}
with the exception of $L[1]$, which is equal to
\[L[1]=\max\left\{0,
          \max_{j\in\{\ell_0+1,\ldots,h\}\setminus I_0}\left\lceil\frac{a_j-k}{j}\right\rceil,
          \max_{j\in\{\ell_0+1,\ldots,h\}\setminus I_0}\left\lceil\frac{b_j-k}{j}\right\rceil\right\}=0.\]
In particular, if the instance of the procedure \texttt{backtrack} for $\ell_0$
makes a recursive call on the line marked \texttt{Call 2},
the inequality $L[\ell_0]\le a\le b\le U[\ell_0]$ implies that
then the values of $a$ and $b$ at that point satisfy
\[\left|aj-a_j\ell_0\right|\le k,\quad
  \left|bj-a_j\ell_0\right|\le k,\quad
  \left|aj-b_j\ell_0\right|\le k \quad\mbox{and}\quad
  \left|bj-b_j\ell_0\right|\le k\]
for every $j\in\{\ell_0+1,\ldots,h\}\setminus I_0$.
Since the same was true in the instances of the procedure \texttt{backtrack} for $\ell\in\{\ell_0+1,\ldots,h\}\setminus I_0$,
it follows that
\begin{equation}
  \left|a_jj'-a_{j'}j\right|\le k,\quad
  \left|b_jj'-b_{j'}j\right|\le k \quad\mbox{and}\quad
  \left|a_jj'-b_{j'}j\right|\le k
\label{eq:abj}  
\end{equation}  
for all $j,j'\in\{\ell_0+1,\ldots,h\}\setminus I_0$.

Suppose that $\ell_0=0$ and consider the following set $Q\subseteq\ZZ^2$:
\[Q=\{(1,0)\}\cup\bigcup_{i\in\{1,\ldots,h\}\setminus I_0}\{z : a_i\le z\le b_i\mbox{ and }\gcd(z,i)=1\}.\]
Since $h\le k$, $|b_i-a_i|\le k/i$ for all $i\in\{1,\ldots,h\}\setminus I_0$, and
\eqref{eq:abj} holds for all $j,j'\in\{1,\ldots,h\}\setminus I_0$,
the set $Q$ is a $k$-nice set.
Note that $a_i\in Q$ and $b_i\in Q$ for every $i\in\{1,\ldots,h\}\setminus I_0$
because $\gcd(a_i,i)=1$, $\gcd(b_i,i)=1$ and $a_i\le b_i$.
The size of $Q$ is equal to
\[1+\sum_{i\in\{1,\ldots,h\}\setminus I_0}M[i][a_i][b_i]=1+N_{>},\]
and so the return value of the procedure \texttt{backtrack} is actually the size of the $k$-nice set $Q$.
We conclude that if the procedure \texttt{backtrack} returns a value $n$ larger than the parameter $N$ that it was called with,
then there exists a $k$-nice set with height $h$ (note that $h\not\in I_0$ and $(a_h,h)\in Q$) that has size $n$.
It follows that if the procedure \texttt{compute(k,h,N)} returns a value $n$ that is larger than $N$,
then there exists a $k$-nice set with height $h$ that has size $n$.

To complete the proof of the lemma,
we need to show that if there exists a $k$-nice set with height $h$ that has $n>N$ elements,
then the return value of the procedure \texttt{compute(k,h,N)} is at least $n$.
Fix a $k$-nice set $Q$ with height $h$ of maximum possible size $n$ and assume that $n>N$.
By Lemma~\ref{lm:xynon},
we may assume that $(1,0)\in Q$, $(0,1)\in Q$ and $Q\subseteq\{0,\ldots,k\}\times\{0,\ldots,h\}$.
Let $I_0$ be the set of those $y\in\{1,\ldots,h\}$ that are not the second coordinate of any point in $Q$, and
for every $i\in\{1,\ldots,h\} \setminus I_0$,
define $a_i$ and $b_i$ to be the minimum and maximum first coordinate of the points in $Q$ with their second coordinate equal to $i$,
respectively.
We inspect the sequence of recursive calls made on the line marked \texttt{Call 1} for $\ell\in I_0$ and
on the line marked \texttt{Call 2} in the loop for $a=a_{\ell}$ and $b=b_{\ell}$ for $\ell\in\{1,\ldots,h\}\setminus I_0$.
Consider the instance for $\ell_0$ in this sequence of the recursive calls.
Note that
the invariants that we observed earlier imply that
$L[i]\le a_i$ and $b_i\le U[i]$ for any $i\in\{1,\ldots,\ell_0\}\setminus I_0$, and
\[N_{>}=\sum_{i\in\{\ell_0+1,\ldots,h\}\setminus I_0}M[i][a_i][b_i]=\left|Q\cap\{0,\ldots,k\}\times\{\ell_0+1,\ldots,h\}\right|.\]
In particular, if $\ell_0=0$,
then the procedure \texttt{backtrack} returns the size of the set $|Q|=n$ on the line marked \texttt{Update N}.

We now consider the case $\ell_0\in I_0$ (note that $h\not\in I_0$ as the height of $Q$ is $h$).
Note that the value of $N'$ in the \texttt{if} condition just before the line marked \texttt{Call 1} is
\begin{align*}
N' & = N_{>}+1+\sum_{i\in\{1,\ldots,\ell_0-1\},L[i]\le U[i]}M_k[i][L[i]][U[i]] \\
   & \ge N_{>}+1+\sum_{i\in\{1,\ldots,\ell_0\}\setminus I_0}M_k[i][L[i]][U[i]] \\
   & \ge N_{>}+1+\sum_{i\in\{1,\ldots,\ell_0\}\setminus I_0}M[i][a_i][b_i] = |Q|.
\end{align*}
Note that we are using $L[i]\le a_i\le b_i\le U[i]$ and $|b_i-a_i|\le k/i$ for $i\in\{1,\ldots,\ell_0\}\setminus I_0$.
In particular, if the recursive call on the line marked \texttt{Call 1} is not made,
then the value of $N$ is already at least $n=|Q|$ and
so the procedure \texttt{compute(k,h,N)} eventually returns a value that is at least $n$.

We next consider the case $\ell_0\in\{1,\ldots,h\}\setminus I_0$ and
consider the iteration of the \texttt{for} cycles for $a=a_{\ell_0}$ and $b=b_{\ell_0}$.
This iteration reaches at the least the \texttt{if} condition above the line marked \texttt{Call 2}
because $L[\ell_0]\le a_{\ell_0}\le b_{\ell_0}\le U[\ell_0]$,
$\gcd(a_{\ell_0},\ell_0)=1$, $\gcd(b_{\ell_0},\ell_0)=1$ and $\left|b_{\ell_0}-a_{\ell_0}\right|\le k/\ell_0$.
Since $L'[i]\le a_i\le b_i\le U'[i]$ for any $i\in\{1,\ldots,\ell_0-1\}\setminus I_0$,
we obtain that the value of $N'$ in the \texttt{if} condition above the line marked \texttt{Call 2} is
\begin{align*}
N' & = N_{>}+1+M[\ell_0][a_{\ell_0}][b_{\ell_0}]+\sum_{i\in\{1,\ldots,\ell_0-1\},L'[i]\le U'[i]}M_k[i][L'[i]][U'[i]]\\
   & \ge N_{>}+1+M[\ell_0][a_{\ell_0}][b_{\ell_0}]+\sum_{i\in\{1,\ldots,\ell_0-1\}\setminus I_0}M_k[i][L'[i]][U'[i]] \\
   & \ge N_{>}+1+M[\ell_0][a_{\ell_0}][b_{\ell_0}]+\sum_{i\in\{1,\ldots,\ell_0-1\}\setminus I_0}M[i][a_i][b_i] = |Q|.
\end{align*}
In particular, if the recursive call on the line marked \texttt{Call 2} is not made,
then the value of $N$ is already at least $n=|Q|$ and
so the procedure \texttt{compute(k,h,N)} eventually returns a value that is at least $n$.

We conclude that
if the whole sequence of recursive calls consisting of those to be made on the line marked \texttt{Call 1} for $\ell\in I_0$ and
on the line marked \texttt{Call 2} in the loop for $a=a_{\ell}$ and $b=b_{\ell}$ for $\ell\in\{1,\ldots,h\}\setminus I_0$
is not made, then the value of $N$ was already at least $n$.
If the whole sequence of the recursive calls is made,
it reaches the line marked \texttt{Update N} in the call with $\ell=0$ and
the procedure \texttt{backtrack} returns $n$.
In either case, the procedure \texttt{compute(k,h,N)} returns a value that is at least $n$.
Since we have shown that if there exists a $k$-nice set with height $h$ that has $n>N$ elements,
then the return value of the procedure \texttt{compute(k,h,N)} is at least $n$,
the proof of the lemma is now complete.
\end{proof}

We are now ready to determine the maximum size of a $k$-nice set for every $k\in\{3,\ldots,1891\}$.

\begin{theorem}
\label{thm:main_compute}
Let $K_0$ be the set containing the $59$ integers listed in Table~\ref{tab:main}.
For every $k\in\{3,\ldots,1891\}\setminus K_0$, the maximum size of $k$-nice set is
\begin{itemize}
\item $k+4$ if $k \bmod 6 = 2$,
\item $k+3$ if $ k \bmod 6\in\{1,3,5\}$, and
\item $k+2$, otherwise.
\end{itemize}
If $k\in K_0\cap\{3,\ldots,1891\}$, then the maximum size of a $k$-nice set
is the value $N(\TT^2,k)$ given in Table~\ref{tab:main}.
\end{theorem}

\begin{proof}
Fix $k\in\{3,\ldots,1891\}$, and let $N$ be the integer defined as
\begin{itemize}
\item $N=k+4$ if $k \bmod 6 = 2$,
\item $N=k+3$ if $ k \bmod 6\in\{1,3,5\}$, and
\item $N=k+2$, otherwise.
\end{itemize}
Note that Lemma~\ref{lm:height123} yields that there exists a $k$-nice set of size $N$.
For every $h\in\{2,\ldots,\lfloor\sqrt{4k/3}\rfloor\}$ such that
Algorithm~\ref{alg:height} returned \emph{not verified}
we executed procedure \texttt{compute(k,h,N)} given as Algorithm~\ref{alg:compute};
the implementation of Algorithm~\ref{alg:compute} is available as
an ancillary file named \verb|Algorithm2.c| with the arXiv version of this manuscript.
Lemmas~\ref{lm:sqroot-alg} and \ref{lm:main_correct} yield that
the maximum of the numbers returned by the procedures is the maximum size of a $k$-nice set.
If $k\not\in K_0$, all procedures returned $N$, and
if $k\in K_0$, the maximum number returned by the procedures
is the value of $N(\TT^2,k)$ in Table~\ref{tab:main}.
\end{proof}

\section{Main result}
\label{sec:main}

We are now ready to prove our main result.
Before we do so, for completeness,
we present an argument on the maximum size of a $1$-nice set and the maximum size of a $2$-nice set.

\begin{proposition}
\label{prop:k12}
The maximum size of a $1$-nice set is $3$ and the maximum size of a $2$-nice is $4$.
\end{proposition}

\begin{proof}
Fix $k\in\{1,2\}$.
Let $Q$ be a $k$-nice set with maximum size and let $h$ the height of $Q$.
By Lemma~\ref{lm:sqroot}, we may assume that $h\le\sqrt{2k}$, i.e.~$h\le k$ (note that $h$ is an integer).
Finally, by Lemma~\ref{lm:xynon},
we may also assume that $(1,0)\in Q$, $(0,1)\in Q$ and $(1,1)\in Q$, and $Q\subseteq \{0,\ldots,k\}^2$.
Hence, if $k=1$, the size of $Q$ is at most $3$ and
the $1$-nice set $\{(1,0), (0,1), (1,1)\}$ witnesses that this bound is tight.
If $k=2$, observe that the set $Q$ cannot contain both $(2,1)$ and $(1,2)$,
which implies that the size of $Q$ is at most $4$.
The $2$-nice set $\{(1,0), (0,1), (1,1), (2,1)\}$ witnesses that this bound is tight.
\end{proof}

We next prove our main result, which is equivalent to Theorem~\ref{thm:main}.

\begin{theorem}
\label{thm:main1}
Let $N_k$ for $k\in\NN$ be the maximum size of a $k$-nice set and
let $K_0$ be the set containing the $59$ integers listed in Table~\ref{tab:main}.
For every $k\in\NN\setminus K _0$, 
it holds that
\[
N_k = \begin{cases}
k+4 &\text{ if } k \bmod 6 = 2,\\
k+3 &\text{ if } k \bmod 6\in\{1,3,5\}, \mbox{and}\\
k+2 &\text{ otherwise}.
\end{cases}
\]
The values of $N_k$ for $k\in K_0$ is the value $N(\TT^2,k)$ given in Table~\ref{tab:main}.
\end{theorem}

\begin{proof}
The values of $N_1$ and $N_2$ are determined by Proposition~\ref{prop:k12} and
the values of $N_k$ for $k\in\{3,\ldots,1891\}$ are determined in Theorem~\ref{thm:main_compute}.
If $k\ge 1892$,
Theorem~\ref{thm:k0new} implies that there exist a maximum size $k$-nice set that has height at most three, and
so the values of $N_k$ for $k\ge 1892$ are determined by Lemma~\ref{lm:height123}.
\end{proof}

Since the proof of Theorem~\ref{thm:main1} is computer assisted,
we next give a weaker version of the theorem, which can be proven without computer assistance.
Note that Theorem~\ref{thm:main2} implies that the maximum size of a $k$-nice set is $k+O(1)$.
In the proof of Theorem~\ref{thm:main2}, we use Lemma~\ref{lm:density}
where we used computer assistance to verify \eqref{eq:sum210},
which requires checking 210 specific inequalities.
In our view, this part of the proof can be verified in a human (although quite tedious) way;
it is also possible to establish Lemma~\ref{lm:density} with the constant $\frac{4946}{3675}$ replaced with $210$
without any estimates requiring a tedious verification as
$\rho_{\ell}\ge 1$ and $\alpha_{\ell}\le 209$ for $\ell\in\{1,\ldots,210\}$.

\begin{theorem}
\label{thm:main2}
Let $N_k$ for $k\in\NN$ be the maximum size of a $k$-nice set.
There exists $k_0$ such that 
it holds for every $k\ge k_0$ that
\[
N_k = \begin{cases}
k+4 &\text{ if } k \bmod 6 = 2,\\
k+3 &\text{ if } k \bmod 6\in\{1,3,5\}, \mbox{and}\\
k+2 &\text{ otherwise}.
\end{cases}
\]
\end{theorem}

\begin{proof}
Let $h_0=41020$, $C=\frac{3264\pi}{10255}$ and $B=\frac{4946}{3675}$.
Note that Lemma~\ref{lm:density} implies that every $k$-nice set with height $h\ge h_0$
has size at most $Ck+Bh+1$.
Further, let
\[\Gamma=\max\{\gamma_4,\ldots,\gamma_{h_0}\}.\]
Note that Lemma~\ref{lm:LP} implies that $\Gamma\in (0,1)$.
Define $k_0$ to be
\[k_0=\max\left\{\frac{2B^2}{(1-C)^2},\frac{\beta_{h_0}-1}{1-\Gamma}\right\}.\]

Fix $k\ge k_0$ and consider a $k$-nice set $Q\subseteq\ZZ^2$ of maximum size.
Note that $|Q|\ge k+2$ as
the set $\{(1,0),(0,1),(1,1),\ldots,(k,1)\}$ is a $k$-nice set of size $k+2$.
By Lemma~\ref{lm:sqroot},
we can assume that $Q$ is $y$-non-negative with height $h\le\sqrt{2k}$.
If $h\ge h_0$,
then Lemma~\ref{lm:density} yields that the size of $Q$ is at most
\[Ck+Bh+1\le Ck+B\sqrt{2k}+1
         \le Ck+\frac{B\sqrt{2}}{\sqrt{k_0}}k+1
	 =Ck+(1-C)k+1=k+1.\]
If $4 \leq h \leq h_0$,
then Lemma~\ref{lm:size} yields that the size of $Q$ is at most
\[\gamma_h k+\beta_h\le \Gamma k+\beta_{h_0}
                    \le \Gamma k+(1-\Gamma)k_0+1\le k+1.\]
Since the size of set $Q$ is at least $k+2$, its height $h$ is at most three.
Hence, the size of $Q$ is determined by Lemma~\ref{lm:height123}.
\end{proof}

We finish by discussing the value $k_0$ from the statement of Theorem~\ref{thm:main2}
when no attempt to optimize any of the parameters is made and no computer assistance is involved;
In particular,
we use Lemma~\ref{lm:density} with $\frac{4946}{3675}$ replaced with $210$ and
we use a straightfoward estimate $\beta_{\ell}\le \ell(\ell+1)$.
Note that the proof of Lemma~\ref{lm:LP} yields that
\[\gamma_{\ell}\le 1 - 2 \left( \frac{1}{\ell-2} - \frac{1}{\ell-1}  \right) \left( \frac{1}{\ell-1} - \frac{1}{\ell}  \right)\]
for every $\ell\ge 4$.
It follows that
\[\Gamma\le 1-\frac{1}{20510\cdot 41019^2\cdot 41018}\le 1-\frac{41020^4},\]
which implies that
\[k_0\le\max\left\{\frac{2\cdot 210^2\cdot 10255^2}{(10255-3264\pi)^2},41020^5\cdot 41021\right\}
	\approx 4.764\cdot 10^{27}.\]
We remark that if the bounds obtained with computer assistance are used,
then the value of $k_0$ as defined in the proof of Theorem~\ref{thm:main2} is $537\,904\,788$.

\section{Conclusion}
\label{sec:concl}

As we have shown,
there are only four values of $k\in\NN$ such that there is a $k$-nice set of size $k+6$:
these are $k=24$, $k=48$, $k=120$ and $k=168$.
We present the sets for these four values of $k$;
we remark that the set $A_{24}$ and $A_{48}$ given below were already identified in~\cite{GasLRRX20}.
The $24$-nice set $A_{24}\subseteq\ZZ^2$ of size $30$
is the following,
\begin{align*}
A_{24}= \{ &
      (1,0),  
      (0,1), (1,1), (2,1), (3,1), (4,1), (5,1), (6,1), (7,1), (8,1), (9,1),\\
      &
      (5,2), (7,2), (9,2), (11,2), (13,2),
      (10,3), (11,3), (13,3), (14,3), (16,3),\\
      & (17,3),
      (15,4), (17,4), (19,4), (21,4),
      (21,5), (22,5), (23,5), (24,5)
      \},
\end{align*}
which can also be described as
\[A_{24}=\{(1,0)\}\cup\bigcup_{i=1,\ldots,5}\{(z,i)  :  5i-5\le z\le 4i+5  \text{ and } \gcd(z,i) = 1\}.\]
The $48$-nice set $A_{48}\subseteq\ZZ^2$ of size $54$,
$120$-nice set $A_{120}\subseteq\ZZ^2$ of size $126$ 
the $168$-nice set $A_{168}\subseteq\ZZ^2$ of size $174$ are as follows:
\begin{align*}
A_{48} & =\{(1,0)\}\cup\bigcup_{i=1,\ldots,7}\{(z,i)  :  7i-7\le z\le 6i+7 \text{ and } \gcd(z,i) = 1\},\\
A_{120} & =\{(1,0)\}\cup\bigcup_{i=1,\ldots,11}\{(z,i)  :  11i-11\le z\le 10i+11 \text{ and } \gcd(z,i) = 1\},\\
A_{168} & =\{(1,0)\}\cup\bigcup_{i=1,\ldots,13}\{(z,i)  :  13i-13\le z\le 12i+13 \text{ and } \gcd(z,i) = 1\}.
\end{align*}
The sets $A_{24}$, $A_{48}$ and $A_{120}$ are visualized in Figures~\ref{fig:k24}, \ref{fig:k48} and~\ref{fig:k120}.

\begin{figure}
\begin{center}
\epsfbox{crtorus-1.mps}
\end{center}
\caption{Visualization of a $24$-nice set with 30 elements presented in Section~\ref{sec:concl}.}
\label{fig:k24}
\end{figure}

\begin{figure}
\begin{center}
\epsfbox{crtorus-3.mps}
\end{center}
\caption{Visualization of a $48$-nice set with 54 elements presented in Section~\ref{sec:concl}.}
\label{fig:k48}
\end{figure}

\begin{figure}
\begin{center}
\epsfbox{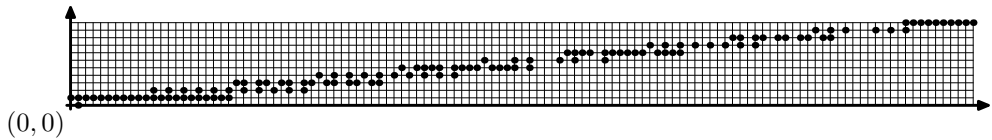}
\end{center}
\caption{Visualization of a $120$-nice set with 126 elements presented in Section~\ref{sec:concl}.}
\label{fig:k120}
\end{figure}

\section*{Acknowledgment}

The authors would like to thank Jordan Ellenberg for insightful comments on the results presented in this paper and their extensions,
particularly for bringing the manuscript~\cite{KriS25} to their attention.
The authors would also like to thank the anonymous referee for very useful comments,
including pointing out some earlier work on the topic.

\bibliographystyle{bibstyle}
\bibliography{crtorus}

@article{GasLRRX20,
  title = "Combinatorics of k-Farey graphs",
  author = "Jonah Gaster and Miguel Lopez and Emily Rexer and Zo{\"e} Riell and Yang Xiao",
  journal = "Rocky Mountain J. Math.",
  volume = "50",
  year = "2020",
  pages = "135--151"
  }

@article{ArtEGOVW16,
  title = "A note on non-degenerate integer programs with small sub-determinants",
  author = "Stephan Artmann and Friedrich Eisenbrand and Christoph Glanzer and Timm Oertel and Santosh Vempala and Robert Weismantel",
  journal = "Operations Research Letters",
  volume = "44",
  year = "2016",
  pages = "635--639"
  }

@inproceedings{ArtWZ17,
  title = "A strongly polynomial algorithm for bimodular integer linear programming",
  author = "Stephan Artmann and Robert Weismantel and Rico Zenklusen",
  booktitle = "Proceedings of the 49th Annual ACM SIGACT Symposium on Theory of Computing (STOC)",
  year = "2017",
  pages = "1206--1219"
  }

@article{AveS24,
  title = "On the maximal number of columns of a $\Delta$-modular integer matrix: bounds and computations",
  author = "Gennadiy Averkov and Matthias Schymura",
  journal = "Mathematical Programming",
  volume = "206",
  year = "2024",
  pages = "61--89"
  }

@article{BonSEHN14,
  title = "On sub-determinants and the diameter of polyhedra",
  author = "Nicolas Bonifas and Marco Di Summa and Friedrich Eisenbrand and Nicolai H{\"a}hnle and Martin Niemeier",
  journal = "Discrete \& Computational Geometry",
  volume = "52",
  year = "2014",
  pages = "102--115"
  }

@inproceedings{FioJWY22,
  title = "nteger programs with bounded subdeterminants and two nonzeros per row",
  author = "Samuel Fiorini and Gwena{\"e}l Joret and Stefan Weltge and Yelena Yuditsky",
  booktitle = "Proceedings of the 62nd Annual Symposium on Foundations of Computer Science (FOCS)",
  year = "2022",
  pages = "13--24"
  }

@article{JiaB22,
  title = "Enumerating integer points in polytopes with bounded subdeterminants",
  author = "Hongyi Jiang and Amitabh Basu",
  journal = "SIAM Journal on Discrete Mathematics",
  volume = "36",
  year = "2022",
  pages = "449--460"
  }

@article{GriMPV18,
  title = "{FPT}-algorithms for some problems related to integer programming",
  author = "D. V. Gribanov and D. S. Malyshev and P. M. Pardalos and S. I. Veselov",
  journal = "Journal of Combinatorial Optimization",
  volume = "35",
  year = "2018",
  pages = "1128--1146"
  }

@article{PaaSWX24,
  title = "On the Column Number and Forbidden Submatrices for $\Delta$-Modular Matrices",
  author = "Joseph Paat and Ingo Stallknecht and Zach Walsh and Luze Xu",
  journal = "SIAM Journal on Discrete Mathematics",
  volume = "38",
  year = "2024",
  pages = "1--18"
  }

@article{KriKS25,
  title = "On the size of integer programs with bounded non-vanishing subdeterminants",
  author = "Bj{\"o}rn Kriepke and Gohar M. Kyureghyan and Matthias Schymura",
  journal = "Journal of Combinatorial Theory, Series A",
  volume = "212",
  year = "2025",
  pages = "106003"
  }

@article{OxlW22,
  title = "2-modular matrices",
  author = "James Oxley and Zach Walsh",
  journal = "SIAM Journal on Discrete Mathematics",
  volume = "36",
  year = "2022",
  pages = "1231--1248"
  }

@incollection{Kar72,
author="Karp, Richard M.",
editor="Miller, Raymond E.
and Thatcher, James W.
and Bohlinger, Jean D.",
title="Reducibility among Combinatorial Problems",
bookTitle="Complexity of Computer Computations, The IBM Research Symposia Series",
year="1972",
publisher="Springer",
pages="85--103",
}

@book{She97,
  author = "Valery N. Shevchenko",
  title = "Qualitative topics in integer linear programming",
  series = "Translations of Mathematical Monographs",
  volume = "156",
  publisher = "American Mathematical Society",
  year = "1997"
  }

@book{Sti12,
  title={Classical topology and combinatorial group theory},
  author={John Stillwell},
  volume={72},
  year={2012},
  publisher={Springer Science \& Business Media}
  }

@misc{BalFKKS24v1,
    title={Curves on the torus with few intersections},
    author = {Balla, I. and Filakovsk\'y, M. and Kielak, B. and Kr\'al', D. and Schlomberg, N.},
    eprint={arXiv:2412:18002v1},
    archivePrefix={arXiv},
    primaryClass={math.CO}
}

@misc{KriS25,
      title={On generic {$\Delta$}-modular integer matrices with two rows},
      author={Bj{\"o}rn Kriepke and Matthias Schymura},
      eprint={arXiv:2502.15394},
      archivePrefix={arXiv},
      primaryClass={math.CO},
      }

@misc{Ell25-blog,
    title="Another example: sets of arcs and loops with bounded intersection",
    author="Jordan Ellenberg",
    eprint="Quomodocumque blog",
    year="December 2025"
    }

@article{Ago00,
author = {Agol, I.},
title = {Bounds on exceptional {D}ehn filling},
volume = {4},
journal = {Geometry \& Topology},
number = {1},
publisher = {MSP},
pages = {431--449},
year = {2000},
doi = {10.2140/gt.2000.4.431},
URL = {https://doi.org/10.2140/gt.2000.4.431}
}

@article{BakGL15,
author={Baker, Kenneth L. and Gordon, Cameron and Luecke, John},
title={Bridge number, Heegaard genus and non-integral Dehn surgery},
volume={367},
journal={Transactions of the American Mathematical Society},
year={2015},
pages={5753--5830}
}

@article{Sch00,
author={Schaller, Paul Schmutz},
title={Mapping class groups of hyperbolic surfaces and automorphism groups of graphs},
volume={122},
journal={Compositio Mathematica},
year={2000},
pages={243--260}
}

@article{PacTT22,
author = {J\'anos Pach and G\'abor Tardos and G\'eza T\'oth},
title = {Crossings between non-homotopic edges},
journal = {Journal of Combinatorial Theory, Series B},
volume = {156},
pages = {389-404},
year = {2022},
}

@article{AouG23, 
author={Aougab, T. and Gaster, J.},
title={Curves on the torus intersecting at most k times},
volume={174},
DOI={10.1017/S0305004122000421},
number={3},
journal={Mathematical Proceedings of the Cambridge Philosophical Society},
year={2023}, 
pages={569--584}}

@article{Aou17,
author = {Aougab, T.},
title = {Curves intersecting exactly once and their dual cube complexes},
journal = {Groups, Geometry and Dynamics},
volume = {11},
pages = {1061--1101},
year = {2017}
}

@article{Aou14,
author = {Aougab, T.},
title = {Constructing large k-systems on surfaces},
journal = {Topology and its Applications},
volume = {176},
pages = {1--9},
year = {2014}}

@article{Aou18,
 author = {Aougab, T},
 title = {Local geometry of the k-curve graph},
 journal = {Transactions of the American Mathematical Society},
 number = {4},
 pages = {2657--2678},
 volume = {370},
 year = {2018}
}

@article{AouBG19,
  author = {Aougab, T. and Biringer, I. and Gaster, J.},
  title = {Packing Curves on Surfaces with Few Intersections},
  journal = {International Mathematics Research Notices},
  volume = {2019},
  number = {16},
  pages = {5205--5217},
  year = {2017},
}

@article{JuvMM96,
  title={Systems of Curves on Surfaces},
  author={Juvan, M. and Malni{\v c}, A. and Mohar, B.},
  journal={Journal of Combinatorial Theory, Series B},
  year={1996},
  volume={68},
  pages={7--22}}

@article{Prz15,
  title={Arcs intersecting at most once},
  author={Przytycki, Piotr},
  journal={Geometric and Functional Analysis},
  volume={25},
  pages={658--670},
  year={2015},
  publisher={Springer}
}

@article{Gre18,
  title={On curves intersecting at most once, {II}},
  author={Greene, Joshua Evan},
  journal={arXiv preprint arXiv:1811.01413},
  year={2018}
}

@article{Gre19,
  title={On loops intersecting at most once},
  author={Greene, Joshua Evan},
  journal={Geometric and Functional Analysis},
  volume={29},
  number={6},
  pages={1828--1843},
  year={2019},
  publisher={Springer}
}

@article{MalRT14,
  title={Topological designs},
  author={Malestein, J. and Rivin, I. and Theran, L.},
  journal={Geometriae Dedicata},
  volume={168},
  number={1},
  pages={221--233},
  year={2014},
  publisher={Springer}
}

@article{BakHP01,
  author = {Baker, R. C. and Harman, G. and Pintz, J.},
  title = {The Difference Between Consecutive Primes, {II}},
  journal = {Proceedings of the London Mathematical Society},
  volume = {83},
  number = {3},
  pages = {532--562},
  year = {2001},
  month = {11},
}

@article{Cra21,
  title={Some theorems concerning prime numbers},
  author={Cram{\'e}r, H.},
  journal={Arkiv f{\"o}r Matematik, Astronomi och Fysik},
  volume={15},
  number={5},
  pages={33},
  year={1921}
}

@article{Gra95,
  title={Harald {C}ram{\'e}r and the distribution of prime numbers},
  author={Granville, A.},
  journal={Scandinavian Actuarial Journal},
  volume={1995},
  number={1},
  pages={12--28},
  year={1995},
  publisher={Taylor \& Francis}
}

@article{OliHP14,
  title={Empirical verification of the even {G}oldbach conjecture and computation of prime gaps up to $4\cdot 10^{18}$},
  author={Oliveira e Silva, T. and Herzog, S. and Pardi, S.},
  journal={Mathematics of Computation},
  volume={83},
  number={288},
  pages={2033--2060},
  year={2014}
}
\end{document}